\documentclass[reqno]{amsart}
\calclayout
\usepackage[foot]{amsaddr}

\usepackage{fancyhdr}

\usepackage[english]{babel}
\usepackage[utf8]{inputenc}
\usepackage[T1]{fontenc}
\usepackage[nointegrals]{wasysym}   
\usepackage{csquotes}
\usepackage{wrapfig}
\usepackage{paralist} 

\usepackage{tikz}
\usepackage{pgfplots}
\pgfplotsset{compat = newest}

\usepackage[hidelinks, plainpages=false]{hyperref}
\usepackage{xcolor}
\hypersetup{
	colorlinks,
	linkcolor={burgundy},
	citecolor={mauvetaupe},
	urlcolor={navyblue}
}

\usepackage[
backend=biber,
sorting=nyt,
backref=true
]{biblatex}
\usepackage{amsmath}
\usepackage{nccmath}
\usepackage{mathtools}
\usepackage{amssymb}
\usepackage{amsthm}
\usepackage{extarrows}
\usepackage{interval}
\intervalconfig{soft open fences}
\usepackage{aligned-overset}
\usepackage{bbm}

\usepackage[shortlabels]{enumitem}
\setlist{align=left, left=1em, parsep=0pt, topsep=-\baselineskip, after=\vspace{\baselineskip}}

\setlist{topsep=.5em, itemsep=0em} 

\usepackage{pbox}

\usepackage{subfig}
\usepackage{graphicx}
\usepackage{float}

\usepackage{xkeyval}
\usepackage{todonotes}
\DeclarePairedDelimiter\ceil{\lceil}{\rceil}
\DeclarePairedDelimiter\floor{\lfloor}{\rfloor}

\newcommand{\fa}{\forall\,}

\newcommand{\indicate}[1]{\mathbbm{1}_{#1}}

\newcommand{\abs}[1]{\left\lvert#1\right\rvert} 

\definecolor{burgundy}{rgb}{0.5, 0.0, 0.13}
\definecolor{darkgreen}{rgb}{0.0, 0.39, 0.0}
\definecolor{darkmagenta}{rgb}{0.55, 0.0, 0.55}
\definecolor{darkorange}{rgb}{1.0, 0.55, 0.0}
\definecolor{darkred}{rgb}{0.55, 0.0, 0.0}
\definecolor{darkgray}{rgb}{0.66, 0.66, 0.66}
\definecolor{lightblue}{rgb}{0.68, 0.85, 0.9}
\definecolor{lightgray}{rgb}{0.83, 0.83, 0.83}
\definecolor{mauve}{rgb}{0.88, 0.69, 1.0}
\definecolor{mauvetaupe}{rgb}{0.57, 0.37, 0.43}
\definecolor{navyblue}{rgb}{0.0, 0.0, 0.5}

\DeclareMathOperator{\dCup}{\mathaccent\cdot\cup}
\newcommand\dSupseteq{\mathrel{\ooalign{$\supseteq$\cr
			\hidewidth\raise.225ex\hbox{$\cdot \ \mkern.5mu$}\cr}}}
\newcommand\dSubseteq{\mathrel{\ooalign{$\subseteq$\cr
			\hidewidth\raise.225ex\hbox{$\cdot \, \mkern.5mu$}\cr}}}

\newcommand{\sumOver}[3]{\sum_{{#1}={#2}}^{#3}}
\newcommand{\prodOver}[3]{\prod_{{#1}={#2}}^{#3}}

\renewcommand{\floor}[1]{\left\lfloor #1 \right \rfloor }
\renewcommand{\ceil}[1]{\left\lceil #1 \right \rceil }

\newcommand{\R}{\mathbb{R}}
\newcommand{\N}{\mathbb{N}}

\DeclareMathOperator{\STS}{{STS}}
\DeclareMathOperator{\ess}{ess}
\DeclareMathOperator{\OPT}{OPT}

\newcommand{\mH}{\mathcal{H}}

\newcommand{\mK}{\mathcal{K}}
\newcommand{\mO}{\mathcal{O}}

\newcommand{\of}[1]{\left(#1\right)}
\newcommand{\paren}[1]{\left(#1\right)}
\newcommand{\ofS}[1]{\left[#1\right]}
\newcommand{\set}[1]{\left\{ #1 \right\}}

\newcommand{\downTo}{\downarrow}

\newcommand\restrict[2]{{
		\left.\kern-\nulldelimiterspace 
		#1 
		\vphantom{\big|} 
		\right|_{#2} 
}}

\newcommand{\comment}[1]{}

\usepackage{multicol}

\allowdisplaybreaks

\theoremstyle{definition}

\newtheorem{theorem}{Theorem}[section] 
\newtheorem{lemma}[theorem]{Lemma}
\newtheorem{corollary}[theorem]{Corollary}
\newtheorem{definition}[theorem]{Definition}

\newtheorem*{observation*}{Observation}
\newtheorem{remark}[theorem]{Remark}

\newtheorem{proposition}[theorem]{Proposition}

\newtheorem{conjecture}[theorem]{Conjecture}

\usepackage{newpxtext}
\usepackage[euler-digits,euler-hat-accent]{eulervm}

\usepackage{listings}

\definecolor{codegreen}{rgb}{0,0.6,0}
\definecolor{codegray}{rgb}{0.5,0.5,0.5}
\definecolor{codepurple}{rgb}{0.58,0,0.82}
\definecolor{backcolour}{rgb}{0.95,0.95,0.92}

\lstdefinestyle{mystyle}{
	backgroundcolor=\color{backcolour},   
	commentstyle=\color{codegreen},
	keywordstyle=\color{magenta},
	numberstyle=\tiny\color{codegray},
	stringstyle=\color{codepurple},
	basicstyle=\ttfamily\footnotesize,
	breakatwhitespace=false,         
	breaklines=true,                 
	captionpos=b,                    
	keepspaces=true,                 
	numbers=left,                    
	numbersep=5pt,                  
	showspaces=false,                
	showstringspaces=false,
	showtabs=false,                  
	tabsize=2
}

\usepackage{orcidlink}

\title{Codegree conditions for (fractional) Steiner triple systems}
\author[Michael Zheng]{Michael Zheng\,\orcidlink{0009-0001-2406-6130}}
\address{
	Department of Mathematics and Computer Science,
	Emory University,
	Atlanta, USA
}
\email{xiangxiang.michael.zheng@emory.edu}
\date{October 6, 2025}

\begin{document}
	\begin{abstract}
		We establish an upper bound on the minimum codegree necessary for the existence of spanning, fractional Steiner triple systems in $3$-uniform hypergraphs. This improves upon a result by Lee in 2023. In particular, together with results from Lee's paper, our results imply that if $n$ is sufficiently large and satisfies some necessary divisibility conditions, then a $3$-uniform, $n$-vertex hypergraph $H$ contains a Steiner triple system if every pair of vertices forms an edge in $H$ with at least $0.8579n$ other vertices.
	\end{abstract}
	
	\maketitle
		
	\section{Introduction}

	Two driving directions in the study of combinatorics have been extremal graph theory and more recently its generalizations to uniform hypergraphs. Broadly speaking, extremal graph theory studies the ``minimal density'' (typically measured by the average degree or minimum degree of the graph) at which certain structures appear in the graph. One family of problems are those of \emph{Dirac-type}, where we are interested at what minimum degree certain spanning structures are contained in the graph. The prototypical example for a Dirac-type result is Dirac's theorem for Hamiltonian cycles: 

\begin{theorem}[Dirac 1952, {\cite{dirac_original}}]
	\label{thm:dirac_cycle}
	Every graph with $n \geq 3$ vertices and minimum degree at least $n / 2$ has a Hamiltonian cycle. 
\end{theorem}

Note that this bound is tight as witnessed by $K_{\floor{n/2}-1, \ceil{n/2}+1}$. As an easy corollary of Theorem \ref{thm:dirac_cycle}, one obtains the optimal minimum degree condition for perfect matchings: 

\begin{corollary}
	\label{thm:dirac_matching}
	Every graph with $n \geq 2$ vertices, $n \equiv 0 \pmod{2}$, and minimum degree at least $n / 2$ has a perfect matching. 
\end{corollary} 

Given these results, it is natural to ask analogous questions in uniform hypergraphs. While various generalizations for cycles in higher uniformities were considered, generalizations of Corollary \ref{thm:dirac_matching} with respect to hypergraphs primarily focused on studying minimum degree or minimum codegree conditions for hypergraph matchings (see for instance \cite{kühn_osthus_matching, rödl_rucinski_szemeredi,rödl_rucinski_szemeredi_co, han_person_schacht,treglown_zhao,khan_matching,khan_four_matching}).

Inspired by Linial's presentation on high-dimensional combinatorics (see \cite{linial}), Lee proposes a different generalization. Namely, going from $2$-uniformity to $3$-uniformity, if the primary property of a perfect matching is that every vertex is covered by exactly one edge, why not consider the structure where instead every pair of vertices gets covered by exactly one hyperedge? 

These structures are known as \emph{Steiner triple systems} and are of independent interest in the theory of \emph{combinatorial designs}. For such a design to exist, the corresponding parameters must satisfy certain divisibility conditions. For instance, for there to be a perfect matching on $n$ vertices, $n$ must be an even number. Similarly, it turns out that Steiner triple systems can only exist for $n \equiv 1,3 \pmod{6}$ and then have $n (n-1) / 6$ edges. 

In \cite{hyunwoo}, Lee establishes the first non-trivial upper bound on the \emph{minimum codegree threshold} for Steiner triple systems in $3$-uniform hypergraphs. In particular, he proves the following \emph{transversal} / \emph{rainbow} result:

\begin{theorem}[Lee 2023, {\cite[Thm.$\ $1.6, Thm.$\ $1.8]{hyunwoo}}]
	For any $\varepsilon > 0$, there exists $n_0 = n_0(\varepsilon)$ such that the following holds for every $n \geq n_0$ satisfying $n \equiv 1, 3 \pmod{6}$: Let $\mH = \set{H_1, \dots, H_{n (n-1)/6}}$ be a family of $3$-uniform hypergraph on $n$ vertices sharing the same vertex set $V$. If the minimum codegree of $H_i$ is at least 
	\begin{equation*}
		\paren{\frac{3+ \sqrt{57}}{12} + \varepsilon} n = \paren{0.879\ldots + \varepsilon } n
	\end{equation*}
	for all $i \in [n (n-1) / 6]$, then there exists a \emph{transversal} Steiner triple system $S$. Specifically, the vertex set of $S$ is $V$ and there exists a bijective function $\varphi \colon E(S) \longrightarrow [n (n-1) / 6]$ such that $e \in E(H_{\varphi(e)})$ for all $e \in E(S)$. 
	\label{thm:main_hyunwoo}
\end{theorem}
\begin{remark}
	Note that this theorem immediately implies the same result for the non-transversal setting since one could take all the hypergraphs in $\mH$ to be the same.
\end{remark}
In this paper, we show the following improvement on Lee's result.
\begin{theorem}
	Let $x^\ast$ be the unique root of the polynomial $p(x) = 8x^3 - 22x^2 + 10x - 1$ in $[0, 1/6]$. Then, Theorem \ref{thm:main_hyunwoo} is true even if the minimum codegree of each $H_i$ is at least 
	\label{thm:main_thesis}
	\begin{equation*}
		\paren{1- x^\ast + \varepsilon} n = \paren{0.8578\ldots + \varepsilon}n.
	\end{equation*}
\end{theorem}

The theorem is established by improving estimates on the so-called \hyperref[def:frac_STS_thresh]{\emph{fractional threshold}} for Steiner triple systems, a parameter on which Lee's proof heavily depends. Though Theorem \ref{thm:main_thesis} provides a better upper bound on the minimum codegree threshold for Steiner triple systems, Lee conjectures that the correct value is actually $3n/4 + \Theta(1)$: 

\begin{conjecture}[Lee 2023, {\cite[Conj.$\ $7.1]{hyunwoo}}]
	There is a constant $n_0\in \N$ and $C \in \R$ such that the following holds for all $n \geq n_0$ satisfying $n \equiv 1,3 \pmod{6}$: Let $\mH = \set{H_1, \dots, H_{n (n-1)/6}}$ be a family of $3$-uniform hypergraph on $n$ vertices sharing the same vertex set $V$. If the minimum codegree of $H_i$ is at least $3n/4 + C$ for all $i \in [n (n-1) / 6]$, then there exists a transversal Steiner triple system $S$. 
	\label{conj:hyunwoo1}
\end{conjecture}
Lee also provides a construction in \cite[Prop.$\ $1.7]{hyunwoo} showing that this would be optimal, even when $H_1 = \dots = H_{n(n-1)/6}$. 

\begin{remark}
	Unlike with perfect hypergraph matchings (see for example \cite{han_person_schacht}), no such condition can be established purely using the minimum degree. Indeed, the $3$-uniform hypergraph $H$ where all possible edges are present except for the ones containing a fixed pair $p$ has minimum degree 
	\begin{equation*}
		\binom{v(H) - 1}{2} - (v(H)-2) = \paren{1 + o(1)} \binom{v(H)}{2},
	\end{equation*}
	but does not contain a Steiner triple system as there is no edge covering $p$. 
\end{remark}
\begin{remark}[Notation and terminology]
	Let $r \geq 2$ and $H$ be an $r$-uniform hypergraph. As a shorthand, we write $v(H) = \abs{V(H)}$ and $e(H) = \abs{E(H)}$. For $p \subseteq V(H)$, we denote by $\deg(p) = \abs{\set{e \in E(H) \colon p \subseteq e}}$ the \emph{degree of $p$}. Furthermore, we will write $\delta_1(H)= \min\set{\deg(v) \colon v \in V(H)}$ and $\delta_{r-1}(H) = \min\set{\deg(p) \colon p \in V(H)^{(r-1)}}$ for the \emph{minimum degree} and \emph{minimum codegree of $H$} respectively. If $r = 2$ and $H$ a graph, then we may write $\delta(H)$ instead of $\delta_1(H)$. For $p \in V(H)^{(r-1)}$, we denote by $N(p) = \set{v \in V(H)\colon p \cup \set{v} \in E(H)}$ the \emph{neighborhood of $p$}. Note that we generally have $\deg(p) = \abs{N(p)}$. As there may be $p \in V(H)^{(r-1)}$ with $N(p) = \emptyset$, we let
	\begin{equation*}
		\delta_{r-1}^{\ess}(H) = \min\set{\deg(p) \colon p \in V(H)^{(r-1)}, \deg(p) > 0}
	\end{equation*}
	be the \emph{essential minimum codegree of $H$}. To talk about the hypergraph formed by $(r-1)$-sets with positive codegree, we define $\partial H = \paren{V(H), \set{p \in V(H)^{(r-1)} \colon \deg(p) > 0}}$ to be the \emph{shadow of $H$}. Lastly, for $s \geq r$, we denote by $\mK_s (H)$ the collection of $s$-cliques in $H$, i.e. $$\mK_s (H)= \set{K \subseteq H\colon H \simeq K_s^{(r)}}.$$
\end{remark}
\subsection{Fractional Steiner triple systems}
\label{subsec:reduct}

Before going into the proof of Theorem \ref{thm:main_thesis}, it is necessary to discuss Lee's results in more detail, notably how he achieved Theorem \ref{thm:main_hyunwoo}. Essentially, Lee reduces the existence of (spanning) Steiner triple systems to its fractional relaxation, and then proves sufficient conditions for the existence of the latter. To make this precise, let us first formally define the following parameters:

\begin{definition}[$\theta_{\STS}$]
	Let $\theta_{\STS}$ be the infimum over all $\delta \in [0,1]$ for which there exists $n_0 \in \N$ such that every family $\mH= \set{H_1, \dots, H_{n(n-1)/6}}$ of $3$-uniform hypergraphs with vertex set $V$, $n = \abs{V}\geq n_0$ and $n \equiv 1,3 \pmod{6}$, with $\delta_2 (H_i) \geq \delta n$ for all $i \in [n(n-1)/6]$ has a transversal Steiner triple system. We will refer to $\theta_{\STS}$ as the \emph{minimum codegree threshold (for Steiner triple systems)}.  
	\label{def:STS_tresh}
\end{definition}

\begin{definition}[Weightings, perfect fractional Steiner triple systems]
	Let $H$ be a $3$-uniform hypergraph. We call $\psi \colon E(H) \longrightarrow \R$ a \emph{weighted subhypergraph of $H$} or simply a \emph{weighting}. Furthermore, for $p \subseteq V(H)$, we define 
	\begin{equation*}
		\deg^\psi(p) = \sum_{e \in E(H) \colon p \subseteq e} \psi(e)
	\end{equation*}
	and refer to it as the \emph{weight of $p$}. Furthermore, we call $\psi$ a \emph{perfect fractional Steiner triple system in $H$} if $\psi(E(H))\subseteq  [0,1]$ and every pair $p \in E(\partial H)$ satisfies $\smash{\deg^\psi(p) = 1}$. 
\end{definition}
Note that our definition of a weighted subhypergraph is more general than in \cite{hyunwoo}. If the host graph $H$ is obvious from the context, we will simply refer to $\psi$ as a fractional Steiner triple system.\footnote{Here, we deviate slightly from the naming convention established in \cite{hyunwoo}.} Apart from now being fractional, another difference is that $\psi$ only assigns pairs with positive codegree in $H$ degree one in $\psi$. Analogous to Steiner triple systems, we may also consider the minimum codegree threshold for fractional Steiner triple systems, which leads to the definition of $\theta_{\STS}^f$ and $\theta_{\STS}^\ast$. 
\begin{definition}[$\theta_{\STS}^f$, $\theta_{\STS}^\ast$, {\cite[Def. 1.5]{hyunwoo}}]
	We define the function $\theta_{\STS}^f\colon [0,1) \longrightarrow [0,1]$ as follows: Let $\theta_{\STS}^f(\varepsilon)$ be the infimum over all $\delta \in [0,1]$ for which there exists $n_0 \in \N$ such that for every $3$-uniform hypergraph $H$ on $n \geq n_0$ vertices with $\delta(\partial H) \geq (1- \varepsilon) (n-1)$ and $\delta_2^{\ess} (H) \geq \delta (n-2)$ contains a perfect fractional Steiner triple system. Furthermore, let $\theta_{\STS}^\ast = \lim_{\varepsilon \downTo 0} \theta_{\STS}^f(\varepsilon)$. We refer to $\theta_{\STS}^\ast$ as the \emph{fractional threshold}.
	\label{def:frac_STS_thresh}
\end{definition}
\begin{remark}
	While $\theta_{\STS}^f$ is obviously monotonically decreasing and thus $\theta_{\STS}^\ast \geq \theta_{\STS}^f(0)$, it is not clear whether equality holds.\footnote{We do conjecture it to be the case, see Conjecture \ref{conj:frac_STS}.} We also note that, in the definition of $\theta_{\STS}^f$, $\partial H$ not necessarily being complete is crucial for Lee's proof to go through. Indeed, iterative absorption is employed and, because we are deleting all edges of already covered pairs, the cover down step will always involve hypergraphs where some pairs have zero codegree. 
\end{remark}

As indicated above, the proof of Theorem \ref{thm:main_hyunwoo} then boils down into two steps.

\begin{theorem}[Lee 2023, {\cite[Simplification of Thm. 1.6]{hyunwoo}}]
	$\theta_{\STS} = \max\set{\theta_{\STS}^\ast, 3/4}$. 
	\label{thm:hyunwoo_reduct}
\end{theorem}
\begin{theorem}[Lee 2023, {\cite[Thm. 1.8]{hyunwoo}}]
	$\theta_{\STS}^f(\varepsilon) \leq (3+ \sqrt{57}) / 12$ for all $\varepsilon \in [0,1)$. 
	\label{thm:hyunwoo_frac}
\end{theorem}

Hence, by Theorem \ref{thm:hyunwoo_reduct}, to show Theorem \ref{thm:main_thesis}, it suffices to improve the upper bound on $\theta_{\STS}^\ast$:
\begin{theorem}
	\label{thm:main_fractional}
	Let $x^\ast$ be the unique root of the polynomial $p(x)=8x^3 -22x^2 + 10x - 1$ in $[0,1/6]$. Then, 
	$\theta_{\STS}^f(\varepsilon) \leq 1 - x^\ast < 0.8579$ for any $\varepsilon \in [0, 1)$. In particular, $\theta_{\STS}^\ast \leq 1 - x^\ast < 0.8579$. 
\end{theorem}

Our proof will closely follow Delcourt and Postle's approach in \cite{delcourt_postle_fractional}, in which they established the best known upper bound on the minimum degree threshold for a fractional $K_3$-decomposition. As such, our proof will substantially differ from Lee's proof of Theorem \ref{thm:hyunwoo_reduct}, where he follows Dross' approach in \cite{dross_fractional}. We discuss this connection to (fractional) $K_3$-decompositions more deeply in Subsection \ref{subsec:related}.

\subsection{Related work}
\label{subsec:related}

Instead of thinking of them as hypergraphs, one can alternatively define Steiner triple systems as partitions of $K_n$ into triangles. These partitions are also called $K_3$-decompositions. Hence, naturally, the problem has some resemblance to the \hyperref[conj:nash-williams]{\emph{Nash-Williams conjecture}} on $K_3$-decompositions: We call a graph $G$ \emph{$K_3$-divisible} if $\abs{E(G)}$ is divisible by three and every vertex has even degree. While it is obvious that being $K_3$-divisible is necessary for a graph $G$ to have a $K_3$-decomposition, the question is again when these conditions are sufficient. 

\begin{conjecture}[Nash-Williams 1970, {\cite{nash_williams}}]
	For all sufficiently large $n \in \N$ the following holds: If $G$ is $K_3$-divisible and has minimum degree at least $3n/4$, then $G$ is $K_3$-decomposable. 
	\label{conj:nash-williams}
\end{conjecture}

Phrased in terms of hypergraphs, both problems can be seen to be of the following type: Given a $3$-uniform hypergraph $H$\footnote{Where the hypergraph $H$ in the Nash-Williams setting is given by interpreting the triangles as triples.} satisfying some properties, does there exist a linear\footnote{Meaning that two edges share at most one vertex.} hypergraph $G \subseteq H$ such that $\partial G = \partial H$?

Seeing this connection, Lee successfully transferred techniques used in works concerning Conjecture \ref{conj:nash-williams} to his setting. Notably, in the proof of Theorem \ref{thm:hyunwoo_reduct}, Lee roughly follows the same approach as in \cite[Thm. 1.2]{minimalist_designs}, in which they show that, asymptotically, it suffices to consider the minimum degree threshold for fractional $K_3$-decompositions for Conjecture \ref{conj:nash-williams}. Furthermore, to get the upper bound on $\theta_{\STS}^\ast$ in Theorem \ref{thm:hyunwoo_frac}, Lee uses Dross' approach from \cite{dross_fractional}, which established a sufficient minimum degree condition for fractional $K_3$-decompositions. There, the problem is essentially reduced to a linear optimization program / maximum flow problem with the fractional Steiner triple system constructed using the maximum flow. This was the best known result for a long time till Dukes and Horsley in \cite{dukes_horsley_fractional} and Delcourt and Postle in \cite{delcourt_postle_fractional} independently improved upon Dross' result. 

We also make use of this parallel by following Delcourt and Postle's approach in \cite{delcourt_postle_fractional}.
\subsection{Organization}
The rest of the paper is organized as follows: In Section \ref{sec:gadgets_weighting}, we introduce a weighting $w_H \colon E(H) \longrightarrow \R$ that satisfies all the properties of a fractional Steiner triple system with respect to the given $3$-uniform hypergraph $H$, except possibly non-negativity. Hence, it suffices to analyze under what conditions $w_H$ is non-negative. In Section \ref{sec:reformulation}, we transform this problem to make the subsequent optimization steps easier. In Section \ref{sec:optimization}, we formally state the non-linear optimization problem at hand and reduce the number of variables till we can finally solve it, thus obtaining Theorem \ref{thm:main_fractional}. We end the paper with a few concluding remarks, among them a concrete conjecture for $\theta_{\STS}^\ast$.

\section{Edge-gadgets and the weighting}
\label{sec:gadgets_weighting}
As mentioned, to prove Theorem \ref{thm:main_fractional}, we closely follow Delcourt and Postle's approach in \cite{delcourt_postle_fractional}. Our discussion of their method will be more expository. Still, for the convenience of the reader, we also repeat proofs that otherwise are analogous to the ones in \cite{delcourt_postle_fractional}.

\subsection{Edge-gadgets}

One of the fundamental ingredients is the usage of the so-called \emph{edge-gadgets} which were first introduced in \cite{barber_fractional_clique}. We naturally modify the definition for our setting. 
\begin{definition}[$\psi_{K,p}$]
	\label{def:gadget}
	Let $H$ be a $3$-uniform hypergraph. For $K \in \mK_5(H)$ and $p \in E(\partial K)$, let $E_j (K, e) = \set{e \in E(K) \colon \abs{e \cap p} = j}$ for all $j \in \set{0,1,2}$. The \emph{edge-gadget of $p$ in $K$} is the function 
	\begin{align*}
		\psi_{K, p} \colon E(H) &\longrightarrow \R \\	
		e &\longmapsto \begin{cases}
			+ \frac{1}{3}, & e \in E_0(K, p)\\[5pt] 
			- \frac{1}{6}, & e \in E_1(K, p)\\[5pt]  
			+ \frac{1}{3}, & e \in E_2(K, p)\\[5pt] 
			0, & \text{otherwise.}
		\end{cases}
	\end{align*}
	\comment{
	\begin{align*}
		\psi_{K, p} \colon E(H) &\mathmakebox[\widthof{${}\longrightarrow{}$}][c]{\longrightarrow} \R \\	
		\psi_{K,p} (e) &\mathmakebox[\widthof{${}\longrightarrow{}$}][c]{=} \begin{cases}
			+ \frac{1}{3}, & e \in E_0(K, p)\\[5pt] 
			- \frac{1}{6}, & e \in E_1(K, p)\\[5pt]  
			+ \frac{1}{3}, & e \in E_2(K, p)\\[5pt] 
			0, & \text{otherwise.}
		\end{cases}
	\end{align*}
	}
\end{definition}

What makes these gadgets so useful is that they allow us to alter the weight of the pair $p$ without changing the weight of other pairs. Indeed, it turns out that $\psi_{K,p}$ acts like an indicator of $p$, meaning that it assigns weight one to $p$ and zero to all the other pairs. However, this comes at the cost of introducing negative weights for the edges. 
\begin{proposition}
	\label{prop:gadget_delta}
	Let $K \in \mK_5(H)$ and $p \in E(\partial K)$. Then we have for all $q \in E(\partial  H)$
	\begin{align*}
		\deg^{\psi_{K,p}}(q) = \indicate{p = q}. 
	\end{align*}
\end{proposition}
\begin{proof}
	Clearly, $\smash{\deg^{\psi_{K,p}}(q) = 0}$ for all $q \in E(\partial  H) \setminus E(\partial K)$. So, consider $q \in E(\partial K)$. 
	
	If $\abs{p \cap q} = 0$, then there is exactly one edge in $E_0(K, p)$ and exactly two edges in $E_1(K, p)$ containing $q$, hence 
	\begin{equation*}
		\deg^{\psi_{K,p}}(q) = \frac{1}{3} - \frac{2}{6} = 0.
	\end{equation*}
	
	If $\abs{p \cap q} = 1$, then there are exactly two edges in $E_1(K,p)$ and exactly one edge in $E_2,(K, p)$ containing $q$, hence 
	\begin{equation*}
		\deg^{\psi_{K,p}}(q) = \frac{-2}{6} + \frac{1}{3} = 0.
	\end{equation*}
	
	Lastly, if $\abs{p \cap q} = 2$, i.e. $p = q$, then there are three edges in $E_2(K,p)$ containing $p$, so
	\begin{equation*}
		\deg^{\psi_{K,p}}(q) = \frac{3}{3} = 1. \qedhere
	\end{equation*} 
\end{proof}	

From Proposition \ref{prop:gadget_delta}, we can immediately construct weightings on the edges that satisfy the pair condition if we have no restrictions on our weights.
\begin{corollary}
	\label{cor:pseudo-sts}
	Let $H$ be a $3$-uniform hypergraph such that every pair $p \in E\of{\partial  H}$ is contained in at least one $5$-clique. Then there exists $w \colon E(H) \longrightarrow \R$ such that 
	\begin{equation*}
		\deg^w(p) = 1
	\end{equation*}
	for all pairs $p\in V\of{\partial  H}$.
\end{corollary}
\begin{proof}
	For every pair $p$, let $K_p$ be a $5$-clique containing $p$. Then, by Proposition \ref{prop:gadget_delta}, 
	\begin{equation*}
		w = \sum_{p \in V(\partial  H)} \psi_{K_p, p} 
	\end{equation*}
	is as desired. 
\end{proof}

\subsection{The weighting}

Before we begin to define our weighting, let us first give a couple more definitions. 
\begin{definition}[{$\mK_r (H, F)$, $\mK_r(H, S)$}]
	For a subhypergraph $F \subseteq H$, $S \subseteq V(H)$ and $r \in \set{3,4,5}$, let 
	\begin{align*}
		\mK_r(H, F) &= \set{K \in \mK_r(H) \colon F \subseteq K}, &
		\mK_r(H, S) &= \mK_r(H, H[S]).
	\end{align*} 
\end{definition}
\begin{definition}[$CN$]
	Given $P \subseteq V(H)^{(2)}$, let the \emph{common co-neighborhood of $P$} equal
	\begin{equation*}
		CN(P) = \bigcap_{p \in P} N(p).
	\end{equation*}
	Furthermore, for $S \subseteq V(H)$ we define $	CN(S) = CN\left(S^{(2)}\right)$.
\end{definition}

For the discussion below, assume for now that $\delta_2^{\ess}(H) > 5n/6$ with $n = v(H) \geq 5$. It is then evident that $$\mK_5(H, p) \neq \emptyset$$ for every pair $p \in E(\partial  H)$. Indeed, by definition, there must be an edge $e \in E(H)$ witnessing $p \in E(\partial  H)$; $e$ can be extended to some tetrahedron $K$ since
\begin{equation*}
	\abs{CN(e)} \geq 3\delta_2^{\ess}(H) - 3 n > 0.
\end{equation*} 
Similarly, $K$ can be extended to a $K_5^{(3)}$ since
\begin{equation*}
	\abs{CN(K)} \geq 5\delta_2^{\ess}(H) - 5n > 6 \cdot \frac{5n}{6} - 5n = 0.
\end{equation*} 

As seen in Corollary \ref{cor:pseudo-sts}, the only constraint not immediately satisfied in our usage of edge-gadgets is non-negativity. Indeed, if every edge has a non-negative weight, then the condition on the pairs already implies that the weight of each edge is at most one, making it a fractional Steiner triple system. For this, we want a more general approach than the one in the proof of Corollary \ref{cor:pseudo-sts}. Namely, to be more flexible, instead of relying on a single edge-gadget per pair, it seems advantageous to distribute the \emph{demand} of the pair over \emph{multiple} edge-gadgets of $p$. Hence, one natural approach would be to make the ansatz
\begin{equation*}
	w = \sum_{p \in E \of{\partial  H}} \sum_{K \in \mK_5(H, p)} \lambda_{K, p} \cdot \psi_{K,p} \tag{$\lambda_{K, p} \in \R$}
\end{equation*}
for our fractional Steiner triple system. One obvious constraint on the scalars is
\begin{equation*}
	\sum_{K \in \mK_5(H, q)} \lambda_{K, q} = 1 \tag{$\star$} \label{eq:cond}
\end{equation*}
for every pair $q \in E(\partial  H)$. Indeed, using Proposition \ref{prop:gadget_delta}, we see that 
\begin{align*}
	\deg^w(q) ={}& \sum_{e \in E(H) \colon q \subseteq e} 	\paren{\sum_{p \in E \of{\partial  H}} \sum_{K \in \mK_5(H, p)} \lambda_{K, p} \cdot \psi_{K,p}(e)} \\
	={}& \sum_{p \in E \of{\partial  H}} \sum_{K \in \mK_5(H, p)} \lambda_{K, p} \cdot \paren{\sum_{e \in E(H) \colon q \subseteq e} \psi_{K, p}(e) } \\
	={}& \sum_{p \in E \of{\partial  H}} \sum_{K \in \mK_5(H, p)} \lambda_{K, p} \cdot \indicate{p = q} \\ 
	={}&  \sum_{K \in \mK_5(H, q)} \lambda_{K, q}.
\end{align*}
From this calculation, it is also clear that this condition is sufficient for every pair to get weight one.

\begin{figure}[htb!]
	\centering
	\includegraphics{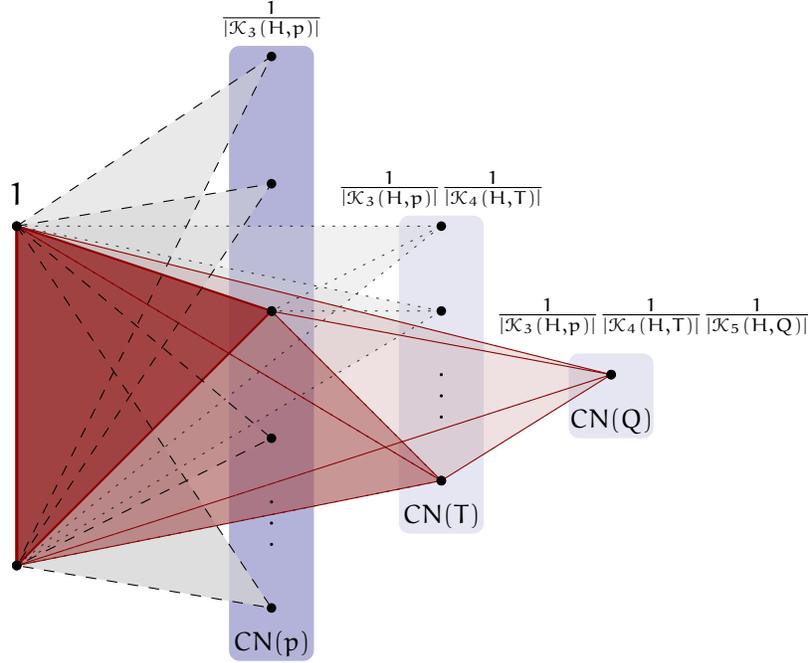}
	\caption{Sketch for how the demand of $p$ is distributed among the $K_5^{(3)}$'s}
\end{figure}
For the scalars, we introduce a non-uniform distribution that utilizes the structure of $H$. Namely, we imagine how at first every pair with positive codegree holds its demand of $1$. Then, each pair \emph{distributes} that demand uniformly among the triples containing that pair. Those triples in turn distribute the demand uniformly among all $\smash{K_4^{(3)}}$'s containing them, until every $\smash{K_5^{(3)}}$ got from every pair a certain fraction of the pair's demand. This fraction will then serve as $\lambda_{K, p}$. 

Note that a $K_4^{(3)}$ containing a pair $p$ will get a fraction of $p$'s demand through two distinct triples, where the fraction received from each triple may be different. Hence, to make this distribution\footnote{Or delegation, as Delcourt and Postle call it.} of the demand formal, it seems natural to introduce \emph{ordered} variants of the previous definitions. 

\begin{definition}[$\mO\mK_r(H)$]
	Let $H$ be a $3$-uniform hypergraph. For $r \in \set{2, 3,4,5}$, an \emph{ordered $r$-clique} of $G$ is an $r$-tuple $K = (v_1, \dots, v_r) \in V(H)^r$ such that $H[\set{v_1, \dots, v_r}] \in \mK_r(H)$.
	
	The vertex set of the ordered clique is $V(K) = \set{v_1, \dots, v_r}$ and the \enquote{$\subseteq$}-relation will be extended in a straightforward way:
	\begin{itemize}
		\item If $K, K'$ are ordered cliques, then $K' \subseteq K$ if $K'$ is a subsequence of $K$. 
		\item If $F$ is a subhypergraph and $K$ an ordered clique, then $F \subseteq K$ holds if $F \subseteq H[V(K)]$.
	\end{itemize}
	The set of ordered $r$-cliques in $H$ is denoted by $\mO\mK_r(H)$. Furthermore, for an ordered $r$-clique $K$ and $s \geq r$, let $\mO\mK_s(H, K)$ denote the set of ordered $s$-cliques containing $K$ in the sense of the \enquote{$\subseteq$}-relation and let $\mO\mK_s(H, S)$ denote the set of ordered $s$-cliques whose vertex set contains $S \subseteq V(H)$. 
\end{definition}

The weight for an ordered clique is now defined as follows: 
\begin{definition}[$W(K)$]
	Let $H$ be a $3$-uniform hypergraph and let $r \in \set{2,3,4}$. For every $K = (v_1, \dots, v_r) \in \mO\mK_r(H)$, we define the weight of $K$ to be 
	\begin{equation*}
		W(K) = \prodOver{i}{2}{r} \frac{1}{\abs{\mK_{i+1} (H, \set{v_1, \dots, v_i})}}.
	\end{equation*} 
	For the sake of clarity, we may also write $W(v_1, \dots, v_r)$ for $W(K)$. 
\end{definition}
For the sake of completeness, we will also define edge-gadgets for ordered cliques. Instead of having the pair explicitly given, the pair for whom the edge-gadget covers the demand is implicitly given by the ordered clique. 
\begin{definition}[$\psi_K$]
	Let $H$ be a $3$-uniform hypergraph. For $K = (v_1, \dots, v_5) \in \mO \mK_5(H)$ and $T \in \mK_3(H)$, we define 
	\begin{equation*}
		\psi_K(T) = \psi_{H[V(K)], v_1v_2} (T).
	\end{equation*} 
	Similarly, we define $\psi_K(O) = \psi_K (H [V(O)])$ for $O \in \mO\mK_3(H)$. 
\end{definition}
While this formalism using ordered cliques gives us a clean way to formally express the distribution of the demand of each pair $p = v_1v_2\in E(\partial H)$, the distribution is now starting from every ordered pair with $(v_1,v_2)$ being distinct from $(v_2,v_1)$. Hence, we define our weighting as follows:
\begin{definition}[$w_H$]
	Let $H$ be a $3$-uniform hypergraph. Define
	\begin{align*}
		w_H \colon E(H) &\longrightarrow \R \\ 
		e &\longmapsto \frac{1}{2} \sum_{K = (v_1,\dots,v_5) \in \mO\mK_5(H, e)} W(v_1, \dots, v_4) \cdot \psi_K (e).
	\end{align*}
\end{definition}
\begin{proposition}
	Let $H$ be a $3$-uniform hypergraph on $n \geq 5$ vertices with $\delta_2^{\ess}(H) > 5n / 6$. Then, the weight of each pair $p \in E(\partial  H)$ equals 
	\begin{equation*}
		\deg^{w_H}(p) = 1.
	\end{equation*}  
	\label{prop:Wh_one}
\end{proposition}
\begin{proof}
	By our introduction of the ansatz, a weighting of this form satisfies the pair condition if and only if (\ref{eq:cond}) holds for every pair. Hence, it suffices to show that 
	\begin{equation*}
		\frac{1}{2} \sum_{\substack{K = (v_1, \dots, v_5) \in \mO\mK_5(H, p)\colon \\ p = v_1 v_2}} W(v_1, \dots, v_4) = 1.
	\end{equation*}
	Indeed, as $W(v_1, \dots, v_4)$ doesn't depend on $v_5$, $W(v_1, v_2, v_3)$ doesn't depend on $v_4$ etc., we compute that 
	\begin{align*}
		\frac{1}{2} \sum_{\substack{K = (v_1, \dots, v_5) \in \mO\mK_5(H, p)\colon \\ p = v_1 v_2}} W(v_1, \dots, v_4) &= \frac{1}{2} \sum_{\substack{K = (v_1, \dots, v_4) \in \mO\mK_4(H, p)\colon \\ p = v_1 v_2}} W(v_1, \dots, v_3) \\
		&= \frac{1}{2} \sum_{\substack{K = (v_1, \dots, v_3) \in \mO\mK_3(H, p)\colon \\ p = v_1 v_2}} W(v_1, v_2) \\
		&= \frac{1}{2} \sum_{\substack{K = (v_1, v_2) \in \mO\mK_2(H, p)\colon \\ p = v_1 v_2}} 1 \\
		&= 1. \qedhere{}
	\end{align*}
\end{proof}
\section{Reformulation}
\label{sec:reformulation}
From this section onwards, the goal is to show that $w_H$ is non-negative and thus a fractional Steiner triple system for sufficiently large (essential) minimum codegree. Instead of showing this directly, however, we consider an ordered variant of the weighting. 

\begin{definition}[$w_H(O)$]
	Let $H$ be a $3$-uniform hypergraph on $n \geq 5$ vertices with $\delta_2^{\ess}(H) > 5n / 6$. For an ordered edge $O \in \mO\mK_3(H)$, let 
	\begin{equation*}
		w_H(O)= \frac{1}{2} \sum_{K = (v_1,\dots,v_5) \in \mO\mK_5(H, O)} W(v_1, \dots, v_4) \cdot \psi_K (O).
	\end{equation*}
\end{definition}

By our extension of the \enquote{$\subseteq$}-relation to ordered cliques, it is evident that 
\begin{equation*}
	w_H(e) = \sum_{O \in \mO\mK_3 (H, e)} w_H(O).
\end{equation*} 
Hence, we may prove the following, stronger result. 
\begin{theorem}
	Let $x^\ast$ be defined as in Theorem \ref{thm:main_fractional} and let $H$ be a $3$-uniform hypergraph satisfying $\delta_2^{\ess}(H) \geq (1 - x^\ast) \cdot v(H)$ and $v(H) \geq 5$. Then, $w_H(O) \geq 0$ for all $O \in \mO\mK_3(H)$. In particular, $w_H$ is a fractional Steiner triple system of $H$.  
\end{theorem}

It will be useful to represent $w_H(O)$ in a more explicit manner.
\begin{lemma}
	\label{lem:cancel}
	Let $H$ be a $3$-uniform hypergraph on $n \geq 5$ vertices with $\delta_2^{\ess}(H) > 5n / 6$. If $O = (x_1, x_2, x_3) \in \mO \mK_3(H)$, then $w_H(O)$ can be written as
	\begin{align*} 
		\frac{1}{6} \bigg( & W(x_1, x_2) - \sum_{y \in CN(x_1, x_2, x_3)} \bigg(W(x_1, y, x_2) - W(x_1, x_2, y)  \ +   \\
		\hspace{30pt}  & \sum_{\mathclap{z \in CN(x_1, x_2, x_3, y)}} \hspace{25pt} \paren{W(x_1, y, x_2, z) - W(x_1, x_2, y, z) + W(x_1, y, z, x_2) - W(y, z, x_1, x_2)} \bigg)\bigg). 
	\end{align*} 
\end{lemma}
In this representation of $w_H(O)$, we can see what Delcourt and Postle refer to as \emph{cancellation}: Apart from the positive leading term $W(x_1, x_2)$, similar weights are paired up to (hopefully) cancel each other in the inner terms of the expression. For example, $W(x_1, y, x_2, z)$ and $W(x_1, x_2, y, z)$ differ in their first product, which are $1/ |{\mK_3(H, \set{x_1, y})}|$ and $1/|\mK_3 (H, \set{x_1, x_2})|$ respectively. Hence, as long as the (essential) minimum codegree is high enough, these terms can't differ too much.  
\begin{proof}[Proof of Lemma \ref{lem:cancel}]
	By explicitly using Definition \ref{def:gadget}, we see that $w_H(O)$ equals 
	\begin{alignat*}{3}
		\frac{1}{2} \sum_{\substack{y \in CN(x_1, x_2, x_3), \\ z \in CN(x_1, x_2, x_3, y)}}  \bigg( &W(x_1, x_2, x_3, y) \cdot\paren{+ \frac{1}{3}} &&+ W(x_1, x_2, y, x_3) \cdot\paren{+ \frac{1}{3}} \\
		&+ W(x_1, y, x_2, x_3) \cdot\paren{- \frac{1}{6}} &&+ W(y, x_1, x_2, x_3) \cdot\paren{- \frac{1}{6}}  \\
		&+ W(x_1, x_2, y, z) \cdot \paren{+ \frac{1}{3}} &&+ W(x_1, y, x_2, z) \cdot\paren{- \frac{1}{6}}\\
		&+ W(y, x_1, x_2, z)\cdot \paren{- \frac{1}{6}} &&+ W(x_1, y, z, x_2) \cdot\paren{- \frac{1}{6}}  \\
		&+ W(y, x_1, z, x_2) \cdot\paren{- \frac{1}{6}} &&+ W(y, z, x_1, x_2) \cdot\paren{+ \frac{1}{3}} \bigg). 
	\end{alignat*}
	
	Recall that the order of the starting pair doesn't change the value of $W(\cdot)$. In particular, we have by that symmetry 
	\begin{align*}
		W(x_1, y, x_2, x_3) &= W(y, x_1, x_2, x_3), \\
		W(x_1, y, x_2, z) &= W(y, x_1, x_2, z), \\
		W(x_1, y, z, x_2) &= W(y, x_1, z, x_2).
	\end{align*} 
	
	Hence, the above expression for $w_H(O)$ can be rewritten as 
	\begin{alignat*}{4}
		\frac{1}{6} \sum_{y \in CN(x_1, x_2, x_3)} \sum_{z \in CN(x_1, x_2, x_3, y)} \bigg({}& W(x_1, x_2, x_3, y)  &&+ W(x_1, x_2, y, x_3)  &&- W(x_1, y, x_2, x_3) \\
		+ W(x_1, x_2, y, z){}& - W(x_1, y, x_2, z) &&- W(x_1, y, z, x_2) 
		&&+ W(y, z, x_1, x_2) \bigg). 
	\end{alignat*}
	
	Note that three of the inner terms do not depend on $z$, so summing over all $z \in CN(x_1, x_2, x_3, y)$ cancels out the last product term in e.g. $W(x_1, x_2, y, x_3)$. 
	
	Thus, we get 
	\begin{alignat*}{3}
		\frac{1}{6} \sum_{y \in CN(x_1, x_2, x_3)}{}& \bigg(W(x_1, x_2, x_3) + W(x_1, x_2, y) - W(x_1, y, x_2)  \\
		+\sum_{z \in CN(x_1, x_2, x_3, y)}{}& \big( W(x_1, x_2, y, z)  - W(x_1, y, x_2, z) - W(x_1, y, z, x_2) + W(y, z, x_1, x_2) \big)\bigg). 
	\end{alignat*} 
	
	Furthermore, note that $W(x_1, x_2, x_3)$ doesn't depend on $y$, so similar manipulations and additional rearranging give us the desired expression for $w_H(O)$. 
\end{proof}

For the optimization step, it will be more convenient to work with the following function:
\begin{definition}[$w_{H,1}$]
	Let $H$ be a $3$-uniform hypergraph on $n \geq 5$ vertices with $\delta_2^{\ess}(H) > 5n / 6$. For $O = (x_1, x_2, x_3) \in \mO \mK_3(H)$, let 
	\begin{align*}
		w_{H,1}(O) ={}& 1 - 6 \abs{\mK_3 (H, \set{x_1, x_2})} w_H(O) \\
		={}&  \abs{CN(x_1, x_2)} \sum_{y \in CN(x_1, x_2, x_3)} \bigg(W(x_1, y, x_2) - W(x_1, x_2, y)  \\
		& + \sum_{z \in CN(x_1, x_2, x_3, y)}  (W(x_1, y, x_2, z) - W(x_1, x_2, y, z) \\
		& + W(x_1, y, z, x_2) - W(y, z, x_1, x_2)) \bigg).
	\end{align*}
	
\end{definition}

Clearly, it then suffices to show the following:

\begin{theorem}
	Let $x^\ast$ be defined as in Theorem \ref{thm:main_fractional} and let $H$ be a $3$-uniform hypergraph satisfying $\delta_2^{\ess}(H) \geq (1 - x^\ast) \cdot v(H)$ and $v(H) \geq 5$. Then, $w_{H,1}(O) \leq 1$ for all $O \in \mO\mK_3(H)$.
\end{theorem}

\section{Optimization}
\label{sec:optimization}
Ideally, our goal can be phrased as follows: Determine 
\begin{equation*}
	\sup\set{d \in \left[0, \frac{1}{6}\right) \colon \lim_{n \longrightarrow\infty}\sup_{\substack{H \ 3\text{-uniform}\colon \\ v(H) \geq n, \\ \delta_2^{\ess}(H) \geq (1-d) v(H)}} \max_{O \in \mK_3(H)} w_{H,1}(O) \leq 1}.
\end{equation*} 

This turns out to still be too difficult, so we will one by one relax the problem. First, to abstract away from $H$, we will rewrite $w_{H,1}$ using \emph{scaled} versions of $CN$ and $W$. 

\begin{definition}[$\widehat{CN}$]
	Let $H$ be a $3$-uniform hypergraph with $V(H) \neq \emptyset$. 
	For $\smash{P\subseteq V(H)^{(2)}}$ and $S \subseteq V(H)$, we define the \emph{common co-neighborhood density} to be 
	\begin{align*}
		\widehat{CN}(P) &= \frac{\abs{CN(P)}}{v(H)} \in [0,1], & \widehat{CN}(S) &= \frac{\abs{CN(S)}}{v(H)} \in [0,1]. 
	\end{align*} 
	Similarly, for an ordered $r$-clique $K = (v_1, \dots, v_r) \in \mO\mK_r(H)$, let 
	\begin{equation*}
		\hat{W}(K) = v(H)^{r-1} W(K) = \prodOver{i}{2}{r} \frac{v(H)}{\abs{\mO\mK_{i+1} (H, \set{v_1, \dots, v_i})}} = \prodOver{i}{2}{r} \frac{1}{\abs{\widehat{CN} (v_1, \dots, v_i)}}.
	\end{equation*}
\end{definition}
The following is then immediate:
\begin{proposition}
	Let $H$ be a $3$-uniform hypergraph with $v(H) \geq 5$ and $\delta_2^{\ess}(H) \geq 5v(H) / 6$. We have
	\begin{align*}
		w_{H,1}(O) ={}&  \abs{\widehat{CN}(x_1, x_2)} \cdot \frac{1}{v(H)}\sum_{y \in CN(x_1, x_2, x_3)} \bigg(\hat{W}(x_1, y, x_2) - \hat{W}(x_1, x_2, y)  \\
		& + \frac{1}{v(H)} \sum_{z \in CN(x_1, x_2, x_3, y)}  (\hat{W}(x_1, y, x_2, z) - \hat{W}(x_1, x_2, y, z) \\
		& + \hat{W}(x_1, y, z, x_2) - \hat{W}(y, z, x_1, x_2)) \bigg).
	\end{align*}
\end{proposition}

In this representation of $w_{H,1}$, all the information of $H$ we really need are the common co-neighborhood densities of a small number of vertex subsets. So, instead of considering $H$ itself, we will work with the common co-neighborhood densities directly using the bounds implied by the properties of $H$. 

To establish these bounds, we review some properties of $\widehat{CN}$: Obviously, $\widehat{CN}$ is monotonically decreasing with respect to $\subseteq$. Furthermore, we have that $\smash{\widehat{CN}}$ is \emph{supermodular}. 
\begin{proposition}
	Let $H$ be a $3$-uniform hypergraph with $V(H) \neq \emptyset$. If $A, B \subseteq V(H)^{(2)}$, then 
	\begin{equation*}
		\widehat{CN}(A \cup B) \geq \widehat{CN} (A) + \widehat{CN} (B) - \widehat{CN}(A \cap B).
	\end{equation*}
\end{proposition}
\begin{proof}
	Set $X = \bigcap_{a \in A} N(a)$ and $Y = \bigcap_{b \in B} N(b)$. Obviously, 
	\begin{equation*}
		\abs{X \cap Y} = \abs{X} + \abs{Y} - \abs{X \cup Y}.
	\end{equation*}
	Note, however, that $\abs{X} = v(H) \cdot \widehat{CN}(A)$ and $\abs{Y} = v(H) \cdot \widehat{CN}(B)$. Moreover, 
	\begin{equation*}
		\abs{X \cap Y} = \abs{\bigcap_{p \in A \cup B} N(p)} = v(H) \cdot \widehat{CN}(A \cup B).
	\end{equation*}
	Similarly, we have 
	\begin{equation*}
		\abs{X \cup Y} = \abs{\paren{\bigcap_{a \in A} N(a)} \cup \paren{\bigcap_{b \in B} N(b)}} \leq  \abs{{\bigcap_{p \in A \cap B} N(p)} } = v(H) \cdot \widehat{CN}(A \cap B).
	\end{equation*}
	Thus, dividing by $v(H)$, we are done. 
\end{proof}

Hence, for vertices $v_1, \dots, v_4 \in V(H)$, where $H$ satisfies the usual assumptions, we have 
\begin{align*}
	\widehat{CN}(v_1, v_2, v_3) &\geq \widehat{CN}(\set{v_1v_2, v_1v_3}) + \widehat{CN}(v_2, v_3) - \widehat{CN}(\emptyset) \\
	&\geq   \paren{\widehat{CN}(v_1, v_2) + \widehat{CN}(v_1, v_3) - \widehat{CN}(\emptyset)} + \widehat{CN}(v_2, v_3) -1 \\
	&\geq \widehat{CN}(v_1, v_2) + \widehat{CN}(v_1, v_3) + \widehat{CN}(v_2, v_3) - 2 \\
	&\geq \widehat{CN}(v_1, v_2) + \widehat{CN}(v_1, v_3) - 1 - d, \\
	\widehat{CN}(v_1, \dots, v_4) &\geq \widehat{CN}\of{\set{v_1, v_2, v_3}^{(2)} \cup \set{v_1, v_2, v_4}^{(2)}} + \widehat{CN}(v_3, v_4) - \widehat{CN}(\emptyset) \\
	&\geq \paren{\widehat{CN}\of{v_1, v_2, v_3} + \widehat{CN}\of{v_1, v_2, v_4} - \widehat{CN}\of{v_1, v_2}} + \widehat{CN}(v_3, v_4) - 1 \\
	&\geq {\widehat{CN}\of{v_1, v_2, v_3} + \widehat{CN}\of{v_1, v_2, v_4} - \widehat{CN}\of{v_1, v_2}} - d.
\end{align*}
So, our goal is now to solve the following program (P1) for $d \in [0, 1/6)$: \\ \\
\emph{Maximize $w_{H,1}(O)$ such that for all $y \in CN(x_1, x_2, x_3)$ and $z \in CN(x_1, x_2, x_3, y)$ }

{\small
	\begin{align*}
		\widehat{CN}(x_1, x_2) &\in [1-d, 1],\\ 
		\widehat{CN}(x_1, y) &\in [1-d, 1], \\
		\widehat{CN}(y, z) &\in [1-d, 1], \\
		\widehat{CN}(x_1, x_2, y) &\in \ofS{\widehat{CN}(x_1, x_2) + \widehat{CN}(x_1, y) - 1- d, \widehat{CN}(x_1, x_2)},  \\
		\widehat{CN}(x_1, y, z) &\in
		\ofS{\widehat{CN}(x_1, y) + \widehat{CN}(y, z) - 1- d, \widehat{CN}(x_1, y)}, \\
		\widehat{CN}(x_1, x_2, y, z) &\in \ofS{\widehat{CN}\of{x_1, x_2, y} + \widehat{CN}\of{x_1, y, z} - \widehat{CN}\of{x_1, y} - d, \widehat{CN}\of{x_1, x_2, y}}. 
	\end{align*}
}%

We note that $w_{H,1}(O)$ is under our (essential) minimum codegree assumption a well-defined, continuous function on the domain of (P1). Indeed, it's clear that all the common co-neighborhood densities are strictly positive and at most $1$.  

To emphasize that we now think in variables, fix an enumeration $y_1, y_2, \dots$ on $CN(x_1, x_2, x_3)$ and $z_{i,1}, z_{i,2},\dots$ on $CN(x_1, x_2, x_3, y_i)$ for every $y_i \in CN(x_1, x_2, x_3)$. Let $R_0 = \abs{{CN}(x_1, x_2, x_3)} $ and $R_i = \abs{{CN}(x_1, x_2, x_3, y_i)}$ for $i \in [R_0]$. We may now relabel the common co-neighborhood densities as follows:
\begin{align*}
	\widehat{CN}(x_1, x_2) &\longrightarrow e_0, \\
	\widehat{CN}(x_1, y_i) &\longrightarrow e_i, \\
	\widehat{CN}(y_i, z_{i,j}) &\longrightarrow f_{i,j},\\
	\widehat{CN}(x_1, x_2, y_i) &\longrightarrow q_{i,0} \text{ for } i \in [R_0],  \\
	\widehat{CN}(x_1, y_i, z_{i,j}) &\longrightarrow q_{i,j} \text{ for } i \in [R_0], j \in [R_i], \\
	\widehat{CN}(x_1, x_2, y_i, z_{i,j}) &\longrightarrow p_{i,j}  \text{ for } i \in [R_0], j \in [R_i].
\end{align*}

Changing the variable names accordingly in $w_{H,1}$ gives 
\begin{align*}
	\hat{W}_1 = \frac{e_0}{v(H)} \sumOver{i}{1}{R_0} \bigg({} \frac{1}{q_{i,0}} \paren{\frac{1}{e_i} - \frac{1}{e_0}} 
	+ \frac{1}{v(H)} \sumOver{j}{1}{R_i} \bigg(\frac{1}{p_{i,j}} \paren{\frac{1}{q_{i,j}} \paren{\frac{1}{e_i} - \frac{1}{f_{i,j}}} + \frac{1}{q_{i,0}} \paren{\frac{1}{e_i} - \frac{1}{e_0}}}\bigg)\bigg).
\end{align*}

Thus, (P1) now reads:\\ \\
\emph{Maximize $\hat{W}_1$ such that for all $i \in [R_0]$ and $j \in [R_i]$ }
\begin{align*}
	e_0 &\in [1-d, 1],
	& e_i &\in [1-d, 1], \\
	f_{i,j} &\in [1-d, 1], 
	& q_{i,0} &\in \ofS{e_0 + e_i - 1- d, e_0},  \\
	q_{i,j} &\in
	\ofS{e_i + f_{i,j} - 1- d, e_i}, 
	& p_{i,j} &\in \ofS{q_{i,0} + q_{i,j} - e_i - d, q_{i,0}}, \\
	R_0 &\in \ofS{\frac{1}{2}\cdot v(H), e_0 \cdot v(H)}, 
	& R_i &\in \ofS{0, q_{i,0} \cdot v(H)}.
\end{align*}

Note that the bounds from $R_0$ and $R_i$ are obtained by considering the properties of $\smash{\widehat{CN}}$ and the fact that $d \in [0, 1/6)$. 

\subsection{Reduction to 8 variables}

To reduce the number of variables involved in (P1), we first employ some symmetrization arguments. 
\begin{lemma}
	The maximum value of (P1) is achieved by a point where for all $i \in [R_0]$ and $j,j' \in [R_i]$, we have 
	\begin{equation*}
		f_{i,j} = f_{i,j'}, q_{i,j} = q_{i,j'}, p_{i,j} = p_{i,j'}.  
	\end{equation*}  
\end{lemma}
\begin{proof}
	Since the domain of (P1) is closed and bounded and $\hat{W}_1$ is well-defined and continuous on the domain of (P1), we find that (P1) has a global maximum. Let $P_0$ be a point that achieves this maximum. For each $i$, let $j_i \in [R_i]$ be the index for which the inner term
	\begin{equation*}
		\frac{1}{p_{i,j}} \paren{\frac{1}{q_{i,j}} \paren{\frac{1}{e_i} - \frac{1}{f_{i,j}}}+ \frac{1}{q_{i,0}} \paren{\frac{1}{e_i} - \frac{1}{e_0}}} 
	\end{equation*}
	is maximized over all $j \in [R_i]$. Then the point $P_0'$ obtained from $P_0$ by setting 
	\begin{align*}
		f_{i,j} &= f_{i,j_i}, \\
		q_{i,j} &= q_{i,j_i}, \\
		p_{i,j} &= p_{i, j_i} 
	\end{align*}
	for all $i \in [R]$ and $j \in[R_i]$ is also a point that achieves this maximum. Moreover, since the constraints for $f_{i,j}$, $q_{i,j}$, $p_{i,j}$ are identical for each $j \in [R_i]$, it follows that $P_0'$ also satisfies the constraints of (P1) as desired.
\end{proof}

Letting $r_i = R_i / v(H)$, $f_{i,j} = f_{i,{j_i}}$, $p_i = p_{i,{j_i}}$, $q_i = q_{i,j_i}$, we form a new program (P2) with a new objective function that has the same optimum value as (P1):
\begin{equation*}
	\hat{W}_2 = \frac{e_0}{v(H)} \sumOver{i}{1}{R_0} \bigg( \frac{1}{q_{i,0}} \paren{\frac{1}{e_i} - \frac{1}{e_0}}  + r_i \bigg(\frac{1}{p_i} \paren{\frac{1}{q_i} \paren{\frac{1}{e_i} - \frac{1}{f_i}} + \frac{1}{q_{i,0}} \paren{\frac{1}{e_i} - \frac{1}{e_0}}}\bigg)\bigg).
\end{equation*}

Thus, (P2) reads:\\ \\
\emph{Maximize $\hat{W}_2$ such that for all $i \in [R_0]$}

\begin{align*}
	e_0 &\in [1-d, 1],
	& e_i &\in [1-d, 1], \\
	f_{i} &\in [1-d, 1] 
	& q_{i,0} &\in \ofS{e_0 + e_i - 1- d, e_0},  \\
	q_{i} &\in
	\ofS{e_i + f_{i} - 1- d, e_i}, 
	& p_{i} &\in \ofS{q_{i,0} + q_{i} - e_i - d, q_{i,0}}, \\
	R_0 &\in \ofS{\frac{1}{2}\cdot v(H), e_0 \cdot v(H)}, 
	& r_i &\in \ofS{0, q_{i,0}}.
\end{align*}
\begin{corollary}
	$\OPT(\text{P1}) = \OPT(\text{P2})$. 
\end{corollary}
\begin{lemma}
	The maximum value of (P2) is achieved by a point where for all $i, i' \in [R_0]$
	\begin{equation*}
		e_i = e_{i'}, f_i = f_{i'}, q_{i,0} = q_{i', 0}, q_i = q_{i'}, p_i = p_{i'}, r_i = r_{i'}. 
	\end{equation*}
\end{lemma}
\begin{proof}
	Again, let $P_0$ be a point that achieves this maximum. Let $i' \in [R_0]$ be the index for which the inner term
	\begin{equation*}
		\frac{1}{q_{i,0}} \paren{\frac{1}{e_i} - \frac{1}{e_0}}  + r_i \bigg(\frac{1}{p_i} \paren{\frac{1}{q_i} \paren{\frac{1}{e_i} - \frac{1}{f_i}} + \frac{1}{q_{i,0}} \paren{\frac{1}{e_i} - \frac{1}{e_0}}}\bigg)
	\end{equation*}
	is maximized over all $i \in [R_0]$. Then the point $P_0'$ obtained from $P_0$ by setting 
	\begin{align*}
		e_{i} &= e_{i'}, 
		&f_i &= f_{i'}, \\
		q_{i,0} &=q_{i', 0}, 
		&q_i &= q_{i'}, \\
		p_i &= p_{i'}, 
		&r_i &= r_{i'}
	\end{align*}
	for all $i \in [R]$ is also a point that achieves this maximum. Moreover, since the constraints for $e_i$, $f_i$, $q_{i,0}$, $q_i$, $p_i $, and $r_i$ are identical for each $i \in [R_0]$, it follows that $P_0'$ also satisfies the constraints of (P2) as desired.
\end{proof}
Letting $r_0 = R_0 / v(H)$, $e = e_{i'}$, $f = f_{i'}$, $q_0 = q_{i',0}$, $q = q_{i'}$, $p = p_{i'}$, $r = r_{i'}$, we form a new program (P3) with a new objective function, but the same optimum value as (P2):
\begin{equation*}
	\hat{W}_3 = e_0 \cdot r_0 \bigg( \frac{1}{q_0} \paren{\frac{1}{e} - \frac{1}{e_0}}  + r \bigg(\frac{1}{p} \paren{\frac{1}{q} \paren{\frac{1}{e} - \frac{1}{f}} + \frac{1}{q_0} \paren{\frac{1}{e} - \frac{1}{e_0}}}\bigg)\bigg).
\end{equation*}

Thus, (P3) reads: \\ \\
\emph{Maximize $\hat{W}_3(e_0, e, f, q_0, q, p, r_0, r)$ subject to}
\begin{align*}
	e_0 &\in [1-d, 1],
	& e &\in [1-d, 1], \\
	f &\in [1-d, 1] 
	& q_0 &\in \ofS{e_0 + e - 1- d, e_0},  \\
	q &\in
	\ofS{e + f - 1- d, e}, 
	& p &\in \ofS{q_0 + q - e - d, q_{0}}, \\
	r_0 &\in \ofS{\frac{1}{2}, e_0}, 
	& r &\in \ofS{0, q_{0}}.
\end{align*}
\begin{corollary}
	$\OPT(\text{P1}) = \OPT(\text{P3})$. 
	\label{cor:8var}
\end{corollary}
At this point, the program (P3) would be small enough to be numerically solved by a commercial solver. 
Such an implementation agreeing with our result can be found in Section \ref{sec:fractional_code}. However, the advantage of our proof is that we get an exact solution for Theorem \ref{thm:main_fractional}. 
\subsection{Reduction to two variables}
\begin{theorem}
	\label{thm:W4}
	$\OPT(\text{P3}) = \OPT(\text{P4})$ where (P4) is defined as follows: 
	\begin{align*}
		\text{\emph{Maximize}}& & \hat{W}_4(e_0, f) &=e_0^2 \paren{\frac{{\frac{1}{1- d} - \frac{1}{e_0}}}{e_0 -2d} + \frac{e_0 - 2d}{e_0 + f - 1 - 4d} \paren{\frac{{\frac{1}{1-d} - \frac{1}{f}}}{f -2d} + \frac{{\frac{1}{1-d} - \frac{1}{e_0}}}{e_0 - 2d}}} \\
		\text{\emph{subject to}}& & e_0 &\in [1-d, 1],\\ 
		& & f &\in [1-d, 1].
	\end{align*}
\end{theorem}

\begin{figure}[htb!]
	\centering
	\includegraphics{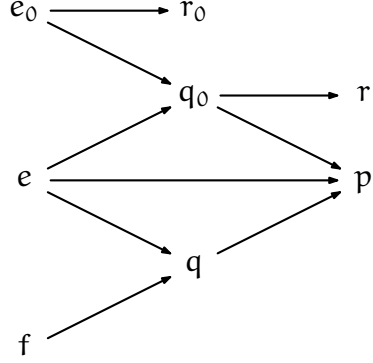}
	\caption{Dependency diagram of the variables' constraints in (P3)}
\end{figure}

\begin{proof}
	For a function $f$, let the \emph{ramp function of $f$} be $f^+ = \max(f, 0)$. Instead of $\hat{W}_3$, we will consider 
	\begin{align*}
		\hat{W}_4 = e_0 \cdot r_0 \left( \frac{1}{q_0} \paren{\frac{1}{e} - \frac{1}{e_0}}^+  + r \left(\frac{1}{p} \paren{\frac{1}{q} \paren{\frac{1}{e} - \frac{1}{f}}^+ + \frac{1}{q_0} \paren{\frac{1}{e} - \frac{1}{e_0}}^+}\right)\right)
	\end{align*}
	as our objective function. Note that $\hat{W}_4$ is also well-defined and continuous on the domain of (P3). In particular, if $e \leq e_0$, $e \leq f$, then $\smash{\hat{W}_4}$ and $\smash{\hat{W}_3}$ correspond for that point. 
	
	Consider first $r_0$. With the ramp functions, it is clear that the term with which $r_0$ is multiplied with is positive, meaning that $\hat{W}_4$ is monotonically increasing if all the remaining variables are fixed. So, to attain the maximum, we must have $r_0 = e_0$. By the same logic, $r = q_0$. To keep track of the substitutions, we let $\sigma = \set{r_0 \longrightarrow e_0, r \longrightarrow q_0}$ be the set of all subsitutions we have made.\footnote{$\sigma$ will be implicitly updated after each subsequent subsitution.} Plugging this in, we get for $\sigma(\hat{W}_4)$
	\begin{equation*}
		e_0^2 \paren{\frac{1}{q_0} \paren{\frac{1}{e} - \frac{1}{e_0}}^+ + q_0 \left(\frac{1}{p} \paren{\frac{1}{q} \paren{\frac{1}{e} - \frac{1}{f}}^+ + \frac{1}{q_0} \paren{\frac{1}{e} - \frac{1}{e_0}}^+}\right)}.
	\end{equation*}
	
	Again, by the same logic, we must have that $p$ is as small as possible, so $p = q_0 + q - e - d$. Updating $\sigma$ accordingly, we get for $\sigma(\hat{W}_4)$
	\begin{equation*}
		e_0^2 \paren{\frac{1}{q_0} \paren{\frac{1}{e} - \frac{1}{e_0}}^+ + q_0 \left(\frac{1}{q_0 + q - e -d} \paren{\frac{1}{q} \paren{\frac{1}{e} - \frac{1}{f}}^+ + \frac{1}{q_0} \paren{\frac{1}{e} - \frac{1}{e_0}}^+}\right)}.
	\end{equation*}
	
	Next, consider $q$. As $q_0 + q - e -d > 0$, we see that, for all other parameters fixed, the functions 
	\begin{align*}
		q &\longmapsto \frac{1}{q_0 + q - e-d}, & q &\longmapsto \paren{\frac{1}{q} \paren{\frac{1}{e} - \frac{1}{f}}^+ + \frac{1}{q_0} \paren{\frac{1}{e} - \frac{1}{e_0}}^+}
	\end{align*}
	are both positive and monotonically decreasing. Hence, we see that $\sigma(\hat{W}_4)$ as a whole is monotonically decreasing in $q$ with all other parameters fixed. Thus, we set $q = e + f - 1 - d$ which yields for $\sigma(\hat{W}_4)$ 
	\begin{equation*} 
		e_0^2 \paren{\frac{1}{q_0} \paren{\frac{1}{e} - \frac{1}{e_0}}^+ + \frac{q_0 }{q_0 + f - 1 -2d} \paren{\frac{\paren{\frac{1}{e} - \frac{1}{f}}^+}{e + f - 1 - d}  + \frac{1}{q_0} \paren{\frac{1}{e} - \frac{1}{e_0}}^+}}. 
	\end{equation*}
	
	Consider now $q_0$. Since $-1/2<-3d\leq f - 1 - 2d \leq -2d <0$ on the domain of $f$ and $$q_0 \geq e_0 + e - 1 - d \geq 1 - 3d > \frac{1}{2}$$ on the domain of $e_0$ and $e$ respectively, we have that 
	\begin{equation*}
		q_0 \longmapsto \frac{q_0}{q_0 + f - 1 - 2d}
	\end{equation*}
	is a positive, monotonically decreasing function for all choices of $e_0$, $e$ and $f$ in their respective domain. As the same thing is obviously true for $q_0 \longmapsto 1/ q_0$, we attain the maximum if $q_0$ is chosen as small as possible, i.e. $q_0 = e_0 + e - 1 - d$. This yields
	\begin{equation*}
		\sigma(\hat{W}_4) = e_0^2 \cdot \paren{\frac{\paren{\frac{1}{e} - \frac{1}{e_0}}^+}{e_0 + e - 1 -d} + \frac{(e_0 + e -  1 - d) \paren{\frac{\paren{\frac{1}{e} - \frac{1}{f}}^+}{e + f - 1 - d} + \frac{\paren{\frac{1}{e} - \frac{1}{e_0}}^+}{e_0 + e - 1 - d}}}{e_0 + e + f - 2 - 3d}}.
	\end{equation*}
	
	Finally, consider $e$. Clearly, by similar types of arguments as before, the functions 
	\begin{align*}
		e &\longmapsto \frac{\paren{\frac{1}{e} - \frac{1}{f}}^+}{e + f - 1 - d},  & e &\longmapsto \frac{\paren{\frac{1}{e} - \frac{1}{e_0}}^+}{e_0 + e - 1 - d}
	\end{align*}
	are non-negative and monotonically decreasing in $e$ for all $e_0, f$ as given by the domain. Furthermore, since $-1/2 < f - 1 - 2 d \leq - 2d < 0$ and $e_0 + e - 1 - d \geq 1 - 3d > 1/2$, the function 
	\begin{equation*}
		e \longmapsto \frac{e_0 + e - 1 - d}{e_0 + e + f - 2 - 3d} = \frac{e_0 + e - 1 - d}{(e_0 + e - 1 -d) + f - 1 - 2d} 
	\end{equation*}
	is positive and monotonically decreasing in $e$ for all $e_0, f$ as given by the domain. 
	
	Hence, we get an optimal solution for $e = 1 - d$. Because of this substitution, the ramp functions are now obsolete and we get 
	\begin{equation*}
		\sigma(\hat{W}_4) = e_0^2 \paren{\frac{{\frac{1}{1- d} - \frac{1}{e_0}}}{e_0 -2d} + \frac{e_0 - 2d}{e_0 + f - 1 - 4d} \paren{\frac{{\frac{1}{1-d} - \frac{1}{f}}}{f -2d} + \frac{{\frac{1}{1-d} - \frac{1}{e_0}}}{e_0 - 2d}}}.
	\end{equation*}
	This concludes the proof. 
\end{proof}
\begin{corollary}
	$\OPT(\text{P1}) = \OPT(\text{P4})$. 
\end{corollary}
\subsection{Reduction to one variable}
\begin{theorem}
	$\OPT(\text{P4}) = \OPT(\text{P5})$ where (P5) is defined as follows: 
	\begin{align*}
		\text{\emph{Maximize}}& & \hat{W}_5(f) &=\frac{\frac{1}{1-d} - 1}{1 - 2d} + \frac{1 - 2d}{f - 4d} \paren{\frac{\frac{1}{1-d} - \frac{1}{f}}{f - 2d} + \frac{\frac{1}{1-d} - 1}{1-2d}} \\
		\text{\emph{subject to}}& & f &\in [1-d, 1].
	\end{align*}
\end{theorem}
\begin{proof}
	To prove the claim, we need to show that the choice $e_0=1$ is optimal. For this, we need to do a somewhat careful analysis. We will show, one by one, that the functions 
	\begin{align*}
		\varphi \colon [1-d,1] \longrightarrow \R_{\geq 0}, e_0 &\longmapsto e_0^2 \cdot  \frac{\frac{1}{1-d} - \frac{1}{e_0}}{e_0 - 2d} = e_0 \cdot \frac{e_0 - (1-d)}{(1-d) (e_0 -2d)}, \\
		\zeta \colon [1-d,1] \longrightarrow \R_{\geq 0}, e_0 &\longmapsto e_0 \cdot \frac{e_0 - 2d}{e_0 + f - 1 - 4d} \cdot \paren{\frac{\frac{1}{1-d} - \frac{1}{f}}{f - 2d} + {\frac{\frac{1}{1-d} - \frac{1}{e_0}}{e_0 - 2d}}}
	\end{align*}
	are monotonically increasing for any fixed $f \in [1-d, 1]$. Then, $e_0 = 1$ would indeed be the optimal choice. 
	
	For $\varphi$, it suffices to show that $\ln(\varphi)\colon [1-d, 1] \longrightarrow [-\infty, \infty)$ is monotonically increasing in $e_0$ for fixed $d \in [0,1/6)$. This gives 
	\begin{equation*}
		\ln(e_0) + \ln(e_0 - 1 + d) - \ln(1-d) - \ln(e_0 - 2d).
	\end{equation*}
	
	Taking the derivative with respect to $e_0$, we get 
	\begin{align*}
		\frac{1}{e_0} + \frac{1}{e_0 - 1 +d} - \frac{1}{e_0 - 2d} &\geq 1 + \frac{1}{1 - 1 + d } - \frac{1}{(1-d) - 2 d} \tag{$e_0 \in [1-d , 1]$}\\
		&\geq 1 + \frac{1}{d} - \frac{1}{1 - 3d} \\
		&> 5. \tag{$d \in [0, 1/6)$}
	\end{align*}
	
	As the derivative is positive, $\ln(\varphi)$ and therefore $\varphi$ is strictly monotonically increasing. 
	
	For $\zeta$, let $1-d \leq x < y \leq 1$. We want to show that 
	\begin{equation*}
		\frac{x (x - 2d)}{x + f - 1 - 4d} \cdot \paren{\frac{\frac{1}{1-d} - \frac{1}{f}}{f - 2d} + {\frac{\frac{1}{1-d} - \frac{1}{x}}{x - 2d}}} 
		\leq 
		\frac{y(y - 2d)}{y + f - 1 - 4d} \cdot \paren{\frac{\frac{1}{1-d} - \frac{1}{f}}{f - 2d} + {\frac{\frac{1}{1-d} - \frac{1}{y}}{y - 2d}}}.
	\end{equation*}
	
	Shuffling terms around, this is equivalent to 
	\begin{align*}
		& \paren{ \frac{x(x - 2d)}{x + f - 1 - 4d} - \frac{y(y - 2d)}{y + f - 1 - 4d}} \cdot \frac{\frac{1}{1-d} - \frac{1}{f}}{f - 2d}\\
		={}&  \paren{\frac{x(x - 2d)}{x + f - 1 - 4d} - \frac{y(y - 2d)}{y + f - 1 - 4d}} \cdot \frac{f - (1-d)}{(1-d)f(f - 2d)} \\
		\leq{}&  \frac{y(y - 2d)}{y + f - 1 - 4d} \cdot \frac{\frac{1}{1-d} - \frac{1}{y}}{y - 2d} - \frac{x(x - 2d)}{x+f - 1 - 4d} \cdot \frac{\frac{1}{1-d} - \frac{1}{x}}{x - 2d} \\
		={}& \frac{1}{1-d} \paren{\frac{y - (1-d)}{y + f -1 -4d} - \frac{x - (1-d)}{x + f - 1 - 4d}}. 
	\end{align*}
	
	We may multiply by $(1-d)(x + f - 1 - 4d)(y + f -1 -4d) > 0$ to get 
	\begin{align*}
		& \paren{x (x-2d) (y + f - 1 -4d) - y (y - 2 d) (x+ f -1 - 4d)} \cdot \frac{f - (1-d)}{f (f-2d)} \\
		={}& \paren{x -  y} \paren{8d^2 + 2d (1-f) - ((1 - f) + 4 d ) (x+y) + xy} \cdot \frac{f - (1-d)}{f (f-2d)} \\
		\leq{}& \paren{y - (1-d)} (x+ f-1 -4d) - (x - (1-d)) (y + f - 1 - 4 d) \\
		={}& (y - x)(f- 1 - 4d) - (1-d) (x - y) \\
		={}& (y- x)(f - 5d) .
	\end{align*} 
	
	Hence, dividing by $(y-x) > 0$ and multiplying by $f (f-2d) > 0$, it suffices to show
	\begin{align*}
		\paren{((1-f) + 4d)(x+y) -  \paren{8d^2 +2d(1-f) +xy}}(f-(1-d)) \leq f (f - 2d) (f - 5d). \tag{$\ast$}
	\end{align*}
	
	Note that 
	\begin{align*}
		& ((1-f) + 4d)(x+y) -  \paren{8d^2 +2d(1-f) +xy} \\
		={}& \underbrace{((1-f) + 4d - x )}_{< 0} y + ((1-f) + 4d) x - 8d^2 -2d(1-f) \\
		\leq{}&  2((1-f) + 4d) x - x^2 - 8d^2 -2d(1-f),
	\end{align*}
	where we used the assumption $y > x$ for the last inequality. Differentiating\footnote{For fixed $d \in [0,1/6)$ and $f \in [1-d, 1]$.} with respect to $x \in [1-d, 1]$, we obtain 
	\begin{align*}
		2 \paren{(1-f) + 4d} - 2 x = 2\paren{1 + 4d - (x+f)} \leq 2 \paren{1 + 4d - 2 (1-d)} = 2 \paren{6d-1} < 0.
	\end{align*}
	
	Thus, since $f - (1-d) \geq 0$, the left hand side of ($\ast$) is bounded by 
	\begin{align*}
		& \paren{2 \paren{(1-f) + 4d}(1-d)  - (1-d)^2 - 8d^2 - 2d (1-f)} \paren{f- (1-d)} \\
		={}& \paren{2 (2d - 1) f + 1 +6d-17d^2} \paren{f -(1-d)}
	\end{align*}
	
	Hence, we may show that $\eta \colon [1-d,1] \longrightarrow \R$ defined by 
	\begin{equation*}
		f \longmapsto  f (f - 2d) (f - 5d) - \paren{2 (2d - 1) f + 1 +6d-17d^2} \paren{f -(1-d)}
	\end{equation*}
	is non-negative for all $d \in [0, 1/6)$. Note that
	$\eta$ expands out to 
	\begin{equation*}
		\eta(f) = f^3 + \paren{2 - 11d} f^2 + \paren{23d^2 - 3} f + \paren{17d^3 - 23d^2 + 5d + 1}
	\end{equation*}
	
	It follows that the derivative of $\eta$ is 
	\begin{align*}
		\eta'(f) &= 3f^2 + \underbrace{ \paren{4-22d}}_{> 0} f + \paren{23d^2 - 3} \\
		&\geq 3 \paren{1-d}^2 + \paren{4-22d} \paren{1-d} + \paren{23d^2 - 3} \\
		&= 4 \paren{1 - 2d}\paren{1 - 6d} \\
		&> 0.
	\end{align*}
	
	As $d < 1/6$, $\eta'(f)$ is positive, meaning that $\eta$ strictly monotonically increasing. Hence, $\eta$ is indeed non-negative as 
	\begin{equation*}
		\eta(1-d) = (1-d) (1-3d) (1-6d) > 0.
	\end{equation*}
	
	Thus, we have shown that $\zeta$ is monotonically increasing in $e_0$, showing that $e_0 = 1$ yields the optimal value for (P4). Plugging that in gives the objective function of (P5). 
\end{proof}
\begin{corollary}
	$\OPT(\text{P1}) = \OPT(\text{P5})$. 
\end{corollary}
\subsection{The final optimization}
Despite reducing our optimization program to only one variable, solving (P5) for general $d \in [0, 1/6)$ remains somewhat tedious, especially by hand. Nonetheless, we will be able to determine the optimal $d$ for which $\OPT(\text{P5}) \leq 1$. In particular, the remainder of this section hopefully clarifies the appearance of $x^\ast$ in Theorem \ref{thm:main_fractional}. To do this, the following proposition will prove to be crucial:
\begin{proposition}
	\label{prop:root}
	The polynomial $p(x) = 8 x^3 - 22x^2 + 10 x - 1$ has a unique root $x^\ast\approx 0.1421657737$ within the interval $[0,1/6]$. In particular, $p(x) < 0$ for $0 \leq x < x^\ast$ and $p(x) > 0$ for $x^\ast < x \leq 1/6$. 
\end{proposition}
\begin{proof}
	To prove the uniqueness, note that for $x \in [0, 1/6]$ 
	\begin{equation*}
		p'(x) = 24 x^2 - 44 x + 10 \geq - \frac{44}{6} + 10 > 0.
	\end{equation*}
	
	Hence, $p$ is strictly, monotonically increasing on the interval $[0,1/6]$. Since, $p(0) = -1$ and $p(1/6) = 5/54 > 0$, this concludes the proof. 
\end{proof}

\begin{lemma}
	\label{lem:W5_iff}
	For $d \in [0,1/6)$, we have  
	\begin{equation*}
		\hat{W}_5(1) \leq 1 \iff d \leq x^\ast.
	\end{equation*} 
	In particular, $\hat{W}_5(1) = 1$ if and only if $d = x^\ast$. 
\end{lemma}
\begin{proof}
	Consider the assignment $f = 1$. We then have 
	\begin{align*}
		\hat{W}_5 (1) 
		= \frac{\frac{1}{1-d} - 1}{1 - 2d} + \frac{1-2d}{1 - 4 d} \paren{\frac{\frac{1}{1-d} - 1}{1 -2d} + \frac{\frac{1}{1-d} - 1}{1-2d}} 
		= \frac{8 d^2 - 3d}{8d^3 - 14d^2 + 7d - 1}. 
	\end{align*}
	
	Note that $8 d^3 - 14 d^2 + 7 d - 1 = (d - 1) (2d-1) (4d-1) < 0$ since $d \in [0, 1/6]$. However, by Proposition \ref{prop:root}, we know that 
	\begin{alignat*}{5}
		d > x^\ast &\iff & \ 8d^3 &> 22d^2 - 10 d + 1 &&\iff & \
		8d^3 -14d^2 + 7d -1 &> 8d^2 -3d \\
		&\iff & 1 &<\frac{8d^2 -3d}{8d^3 -14d^2 + 7d -1} &&\iff & \hat{W}_5(1) &> 1, 
	\end{alignat*}
	where we used $8 d^3 - 14 d^2 + 7 d - 1 < 0$ for the third equivalence. In particular, these equivalences also hold if every \enquote{$<$} and \enquote{$>$} is replaced by \enquote{$=$}. 
\end{proof}
\begin{corollary}
	We have $\OPT(\text{P5}) > 1$ for all $d > x^\ast$.
\end{corollary}
\begin{remark}
	It is interesting to note that $f = 1$ doesn't yield the maximum for $d > x^\ast$. This is not too hard to show and can be seen in Figure \ref{fig:W5}. However, for $d \leq x^\ast$ this assignment is indeed optimal.  
\end{remark}
\newcommand{\maxD}{38} 

\begin{figure}
	\centering
	\begin{tikzpicture}
		\begin{axis}[
			domain=0.8:1,
			yrange=0:2,
			xlabel=$f$,
			ylabel={$\hat{W}_5(f)$}
			]
			\pgfplotsforeachungrouped \D in {1, ..., \maxD}
			{   
				\pgfmathtruncatemacro{\i}{\D / \maxD * 100}
				\edef\temp{\noexpand%
					\addplot [thick, domain=(1 - (\D / (6 * \maxD))):1, color=darkorange!\i!navyblue] {(1/(1 - (\D / (6 * \maxD))) - 1)/(1 - 2*(\D / (6 * \maxD))) + (1 - 2*(\D / (6 * \maxD)))*((1/(1 - (\D / (6 * \maxD))) - 1/x)/(-2*(\D / (6 * \maxD)) + x) + (1/(1 - (\D / (6 * \maxD))) - 1)/(1 - 2*(\D / (6 * \maxD))))/(-4*(\D / (6 * \maxD)) + x)};
				}\temp
			}
			\addplot [ultra thick, domain=0.8578342263:1, color=burgundy]{0.2315686746 + 0.7156684526*((1.165726395 - 1/x)/(-0.2843315474 + x) + 0.2315686746)/(-0.5686630948 + x)};
		\end{axis}
	\end{tikzpicture}
	\caption{$\hat{W}_5(f)$ for $\textcolor{navyblue}{d=0}$ to $\textcolor{darkorange}{d=1/6}$ and $\textcolor{burgundy}{d=x^\ast}$}
	\label{fig:W5}
\end{figure}
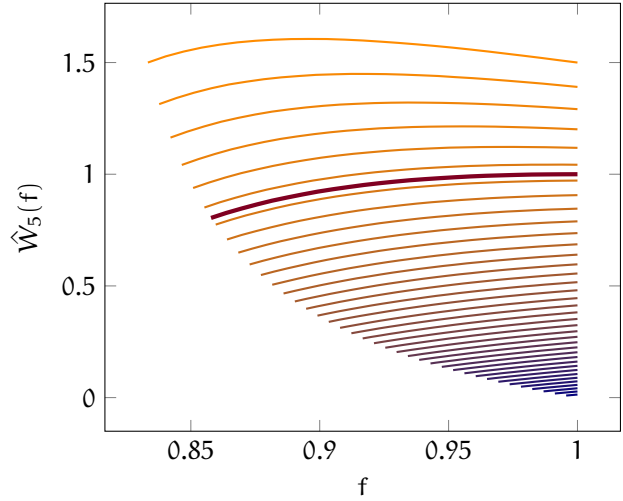
\begin{theorem}
	For $d \leq x^\ast$, $\OPT(\text{P5})$ is obtained by $f = 1$. In particular, 
	\begin{equation*}
		\OPT(\text{P5}) \leq 1
	\end{equation*}
	where equality holds if only if $d = x^\ast$. 
\end{theorem}
\begin{proof}
	We first expand $\hat{W}_5$:
	\begin{align*}
		\hat{W}_5 &= \frac{\frac{1}{1-d} - 1}{1 - 2d} + \frac{1 - 2d}{f - 4d} \paren{\frac{\frac{1}{1-d} - \frac{1}{f}}{f - 2d} + \frac{\frac{1}{1-d} - 1}{1-2d}} \\
		&= \frac{d}{(1-d) (1-2d)} + \frac{1-2d}{1-d} \frac{1}{f-2d} \frac{1}{f- 4d} - (1-2d) \frac{1}{f} \frac{1}{f-2d} \frac{1}{f-4d} + \frac{d}{1-d} \frac{1}{f-4d}.
	\end{align*}
	
	To finish the proof, we only need to show that $\hat{W}_5$ is monotonically increasing in $f$ for fixed $d \in [0,x^\ast]$. The claim then follows from Lemma \ref{lem:W5_iff}. To do so, we will (up to a positive constant) calculate the derivative of $\hat{W}_5$ with respect to $f$ and show that it is non-negative. For that, we will show that actually $\hat{W}_5'$ is monotonically decreasing in $f$ for fixed $d \in [0,x^\ast]$. However, it will turn out that $\hat{W}_5'(1) \geq 0$, completing the proof:
	
	Multiplying $\hat{W}_5$ by $(1-d) / (1-2d) > 0$ and ignoring constant terms, consider 
	\begin{equation*}
		\varphi \colon [1-d, 1] \longrightarrow \R, f \longmapsto \frac{d}{1-2d} \frac{1}{f-4d} + \frac{1}{f-2d} \frac{1}{f-4d} - \frac{1-d}{f} \frac{1}{f-2d} \frac{1}{f-4d}.
	\end{equation*}
	
	The derivative of $\varphi$ is 
	\begin{align*}
		\varphi'(f) ={}& \frac{-d}{1-2d} \frac{1}{(f-4d)^2} - \frac{1}{(f-2d)^2} \frac{1}{f-4d} - \frac{1}{f-2d} \frac{1}{(f-4d)^2} \\
		& + \frac{1-d}{f^2} \frac{1}{f-2d} \frac{1}{f-4d} + \frac{1-d}{f} \frac{1}{(f-2d)^2} \frac{1}{f-4d} + \frac{1-d}{f} \frac{1}{f-2d}\frac{1}{(f-4d)^2}.
	\end{align*}
	
	We want to show that $\varphi'$ is monotonically decreasing in $f$ for fixed $d \in [0, x^\ast]$. Consider the functions
	\begin{align*}
		\psi_1\colon [1-d, 1] \longrightarrow \R_{\geq 0}, f &\longmapsto \paren{\frac{f- (1-d)}{f}} \frac{1}{f- 2d} \frac{1}{(f-4d)^2}, \\
		\psi_2\colon [1-d, 1] \longrightarrow \R_{\geq 0}, f &\longmapsto \paren{\frac{f - (1-d)}{f}} \frac{1}{(f- 2d)^2} \frac{1}{f-4d}, \\
		\psi_3\colon [1-d, 1] \longrightarrow \R_{\geq 0}, f &\longmapsto \frac{1}{f-4d} \paren{\frac{d}{1-2d} \frac{1}{f-4d} - \frac{1-d}{f^2} \frac{1}{f-2d}}.
	\end{align*}
	
	Note that $\varphi' =- (\psi_1 + \psi_2 + \psi_3)$. Hence, it suffices to show that each $\psi_i$ is monotonically increasing in $f$.
	
	For $\psi_1$, we will equivalently show that
	\begin{equation*}
		(\ln(\psi_1))(f) = \ln(f-(1-d)) - \ln(f) - \ln(f-2d) - 2\ln(f-4d)
	\end{equation*}
	is monotonically increasing. For that, we take the derivative:
	\begin{align*}
		(\ln(\psi_1))' (f) &= \frac{1}{f - (1-d)} - \frac{1}{f} - \frac{1}{f-2d} - \frac{2}{f-4d}.
	\end{align*}
	
	To show that the derivative is non-negative, it suffices to show
	\begin{alignat*}{3}
		& & \frac{1}{f - (1-d)} &\geq \frac{1}{f} + \frac{1}{f-2d} + \frac{2}{f-4d} \\ 
		\iff{}& & 1 &\geq \frac{f-(1-d)}{f} + \frac{f-(1-d)}{f-2d} + \frac{2(f- (1-d))}{f-4d} \\
		& &&= \frac{f - (1-d)}{f} + \frac{(f-2d) - (1-3d)}{f-2d} + 2 \cdot \frac{(f- 4d) - (1-5d)}{f-4d}.
	\end{alignat*}
	Clearly, the last expression is monotonically increasing in $f$, meaning that w.l.o.g. we may consider $f = 1$. This gives 
	\begin{alignat*}{4}
		& & 1 &\geq d + \frac{d}{1-2d} + \frac{2d}{1-4d} \\
		&\iff & 8d^2  -6 d +1 &\geq 8d^3 -14d^2 + 4d \\
		&\iff & 0 &\geq 8d^3 -22d^2 + 10d -1.
	\end{alignat*}
	As the last inequality holds by Proposition \ref{prop:root}, the derivative of $\ln(\psi_1)$ is indeed non-negative and we are done. 
	
	We proceed similarly for $\psi_2$: 
	\begin{align*}
		(\ln(\psi_2))(f) &= \ln(f-(1-d)) - \ln(f) - 2\ln(f-2d) - \ln(f-4d), \\
		(\ln(\psi_2))' (f) &= \frac{1}{f - (1-d)} - \frac{1}{f} - \frac{2}{f-2d} - \frac{1}{f-4d}.
	\end{align*}
	
	Comparing $(\ln(\psi_2))'$ with $(\ln(\psi_1))'$, we see that $(\ln(\psi_2))' \geq (\ln(\psi_1))'\geq 0$, giving us immediately that $\psi_2$ is monotonically increasing. 
	
	For $\psi_3$, we equivalently show that $$-\psi_3 (f) = \frac{1}{f-4d} \paren{\frac{1-d}{f^2} \frac{1}{f-2d} - \frac{d}{1-2d} \frac{1}{f-4d}}$$
	is monotonically decreasing. Actually, we will show that  
	\begin{equation*}
		\zeta(f) = \frac{1-d}{f^2} \frac{1}{f-2d} - \frac{d}{1-2d} \frac{1}{f-4d}
	\end{equation*}
	is monotonically decreasing in $f$ for fixed $d \in [0, x^\ast]$. Indeed, that would imply that 
	\begin{align*}
		\frac{1-d}{f^2} \frac{1}{f-2d} - \frac{d}{1-2d} \frac{1}{f-4d} &\geq \frac{1-d}{1-2d} - \frac{d}{1-2d} \frac{1}{1-4d} \\
		&= \frac{1}{1-2d} \paren{1-d - \frac{d}{1-4d}} \\ 
		&> \frac{1}{1-2d} \paren{1 - d - 3d} \tag{$1- 4d > 1/3$} \\
		&> 0,
	\end{align*}
	meaning that $\zeta$ is also positive. 
	As $f \longmapsto 1 / (f-4d)$ is also a positive, monotonically decreasing function, it would follow that $-\psi_3$ is monotonically decreasing. 
	
	Hence, consider $1 - d \leq f < g \leq 1$. We need to show that 
	\begin{alignat*}{3}
		& & \frac{1-d}{f^2 (f-2d)} - \frac{d}{1-2d} \frac{1}{f-4d} &\geq \frac{1-d}{g^2 (g-2d)} - \frac{d}{1-2d} \frac{1}{g -4d} \\
		&\iff & (1-d) \paren{\frac{1}{f^2 (f-2d)} - \frac{1}{g^2 (g-2d)}} &\geq \frac{d}{1-2d} \paren{\frac{1}{f-4d} - \frac{1}{g-4d}}.
	\end{alignat*}
	
	Note that 
	\begin{alignat*}{2}
		\frac{1}{f^2 (f-2d)} - \frac{1}{g^2 (g-2d)} &= \frac{g^2 (g-2d) - f^2 (f-2d)}{f^2 (f-2d) g^2 (g-2d)} 
		&&= \frac{(g-f) \paren{f^2 + fg +g^2 - 2d (g+f)}}{f^2(f-2d) g^2 (g-2d)}, \\
		\frac{1}{f-4d} - \frac{1}{g-4d} &= \frac{(g-4d) - (f-4d)}{(f-4d) (g-4d)} 
		&&= \frac{g-f}{(f-4d) (g-4d)}.
	\end{alignat*}
	
	Thus, it suffices to show that 
	\begin{alignat*}{3}
		& & (1-d) \frac{(g-f) \paren{f^2 + fg +g^2 - 2d (g+f)}}{f^2(f-2d) g^2 (g-2d)} &\geq \frac{d}{1-2d} \frac{g-f}{(f-4d) (g-4d)} \\
		&\iff & \frac{f^2 + fg +g^2 - 2d (g+f)}{f^2 g^2} &\geq \frac{d}{(1-d)(1-2d)} \frac{f-2d}{f-4d} \frac{g-2d}{g-4d}.
	\end{alignat*}
	
	We will lower bound and upper bound the left hand side and right hand side respectively:
	\begin{align*}
		\frac{d}{(1-d)(1-2d)} \frac{f-2d}{f-4d} \frac{g-2d}{g-4d} &\leq \frac{d}{(1-d) (1-2d)} \frac{(1-3d)^2}{(1-5d)^2} \\
		&= \frac{d}{(1-d) (1-2d)} \paren{1 + \frac{2d}{1-5d}}^2, \\
		\frac{f^2 + fg +g^2 - 2d (g+f)}{f^2 g^2} &\geq \frac{f^2 + fg +g^2 - 2d (g+f)}{f^2} \\
		&= 1 + \frac{g}{f} + \paren{\frac{g}{f}}^2 - 2d \paren{\frac{g}{f^2} + \frac{1}{f}} \\
		&\geq 3 - 2 d \paren{\frac{1}{(1-d)^2} + \frac{1}{1-d}},
	\end{align*}
	where we used $f^2 + fg +g^2 - 2d (g+f) > 0$ for the second inequality. 
	
	Note that the lower bound is monotonically decreasing in $d$ and the upper bound is monotonically increasing in $d$. Hence, 
	\begin{align*}
		& \frac{f^2 + fg +g^2 - 2d (g+f)}{f^2 g^2} - \frac{d}{(1-d)(1-2d)} \frac{f-2d}{f-4d} \frac{g-2d}{g-4d} \\
		\geq{}& 3 - 2 d \paren{\frac{1}{(1-d)^2} + \frac{1}{1-d}} - \frac{d}{(1-d) (1-2d)} \paren{1 + \frac{2d}{1-5d}}^2 \\
		\geq{}& 3 - \frac{8}{25} \paren{\paren{\frac{25}{21}}^2 + \frac{25}{21}} - \frac{4}{21} \frac{25}{17} \paren{1 + \frac{8}{5}}^2 \tag{$d \leq x^\ast < 4/25$} \\
		\geq{}& \frac{2039}{7497} \\
		>{}& 0.
	\end{align*}
	
	Thus, $\zeta$ is monotonically decreasing, meaning that $\psi_3$ is monotonically increasing.
	
	Altogether, we therefore have that $\varphi'$, which is up to a positive constant $\hat{W}_5'$, is indeed monotonically decreasing in $f$ for fixed $d \in [0, x^\ast]$. It remains to show that $\varphi'(1) \geq 0$:
	\begin{align*}
		- \psi_1 (1) &= -  \frac{d}{1-2d} \frac{1}{(1-4d)^2}, \\
		- \psi_2 (1) &= - \frac{d}{(1-2d)^2} \frac{1}{1-4d}, \\ 
		- \psi_3 (1) &= \frac{1}{1-4d} \paren{\frac{1-d}{1-2d} - \frac{d}{1-2d} \frac{1}{1-4d}},\\
		\varphi'(1) &= \frac{1}{(1-2d)(1-4d)} \paren{(1-d) - d \paren{\frac{1}{1-2d} + \frac{2}{1-4d} }} \\
		&= \frac{1}{(1-2d)^2(1-4d)^2} \paren{(1-d) (1-2d) (1-4d)- d \paren{(1-4d) + 2 (1-2d)}} \\
		&= \frac{-8d^3 +22d^2 -10d+1}{(1-2d)^2 (1-4d)^2} \\
		&\geq 0. \tag{Prop. \ref{prop:root}}
	\end{align*}
	This completes the proof. 
\end{proof}

\section{Concluding remarks}
\label{sec:conclusions}

	In this paper, we improved the upper bound on the fractional threshold given by Lee in Theorem \ref{thm:hyunwoo_frac} and thus also improved the previous upper bound on the minimum codegree threshold $\theta_{\STS}$ for Steiner triple systems. Furthermore, we also showed that, without any modifications, this is the best achievable result using Delcourt and Postle's approach in \cite{delcourt_postle_fractional}. As such, it seems unlikely that it can be used to (asymptotically) resolve Conjecture \ref{conj:hyunwoo1}. Indeed, in \cite[Rem. 3.11]{ravry_zhao}, it is shown that the minimum codegree required for every vertex to be in some $\smash{K_5^{(3)}}$ is at least $3n/4 + \Theta(1)$. Furthermore, with the standard methods used in this paper, for every pair of positive codegree to be contained in a clique of order five required the essential minimum codegree to be greater than $5 n / 6$. This is in contrast to the setting in \cite{delcourt_postle_fractional}, where a minimum degree greater than $3n/4$ suffices for every pair to be in a $K_5$. In particular, our application of Delcourt and Postle's method unavoidably achieved a weaker bound in comparison to their original result in \cite{delcourt_postle_fractional}. Surprisingly, it was the opposite for Lee's application of Dross's approach: Whereas Dross derived an upper bound of $9n/10$, Theorem \ref{thm:hyunwoo_frac} yields an upper bound of roughly $0.88n$. 

	One aspect unique to our setting is that we not only have control on the minimum essential codegree of $H$, but also the minimum degree in $\partial H$. This was utilized by neither Lee nor by the method presented here. Hence, it seems likely that any improvement on Theorem \ref{thm:main_fractional} will have to make use of this extra constraint. 
\subsection{The conjectured value of $\theta_{\STS}^\ast$}

Naturally, the most immediate direction for further research is showing that $\theta_{\STS}^\ast \leq 3/4$. Together with Theorem \ref{thm:hyunwoo_reduct}, this would resolve the asymptotic version of Conjecture \ref{conj:hyunwoo1}. However, we conjecture that the value of $\theta_{\STS}^\ast$ is actually strictly smaller than $\theta_{\STS}$. Indeed, during his research, Lee originally thought that $\theta_{\STS}$ equals $2/3$ and now believes that:

\begin{conjecture}[Lee 2023, {\cite{lee_personal}}]
	$\theta_{\STS}^\ast = 2/3$.
	\label{conj:frac_STS}
\end{conjecture}

This is motivated by the following natural construction.

\begin{proposition}[Lee 2023, {\cite{lee_personal}}]
	Consider the $3$-uniform hypergraph $H = (V, E)$ on $n \geq 5$ vertices which is constructed as follows: Take an equitable\footnote{Meaning the size of any two partition classes differ by at most by one.} partition $V_1 \dCup V_2 \dCup V_3$ of $V$ and set
	\begin{equation*}
		\smash{E = V^{(3)} \setminus \set{e \in V^{(3)} \colon \abs{e \cap V_1} = \abs{e \cap V_2} = \abs{e \cap V_3} = 1}}.
	\end{equation*}
	Then $\delta_2(H)\geq 2n /3 - 8/3$ and $H$ doesn't contain a fractional Steiner triple system.
	\label{prop:cons2}
\end{proposition}
\begin{proof}
	Since the partition is equitable, we have for all $i \in [3]$
	\begin{equation*}
		\frac{n}{3} - \frac{2}{3} \leq \abs{V_i} \leq \frac{n}{3} + \frac{2}{3}.
	\end{equation*} 
	
	Furthermore, by construction, it follows that the codegree is smallest if the vertices are taken from two different parts. Hence, we get 
	\begin{equation*}
		\delta_2 (H) = (\abs{V} - 2) - \max_{i \in [3]} \abs{V_i} \geq \frac{2n}{3} - \frac{8}{3}. 
	\end{equation*}
	
	It remains to show that $H$ doesn't contain a fractional Steiner triple system: Assume that $\varphi$ is a fractional Steiner triple system in $H$. By definition, we have 
	\begin{align*}
		3 \sum_{e \in E(H)} \varphi(e) = \sum_{e \in E(H)} \sum_{p \in E(\partial H) \colon p \subseteq e} \varphi(e) 
		= \sum_{p \in E(\partial H)} \deg^\varphi(p)
		= \sum_{p \in E(\partial H)} 1 
		= \binom{n}{2},
	\end{align*}
	where we have used that every edge covers exactly $3$ pairs and that $\partial H$ is complete. 
	
	For concreteness, let $n_i = \abs{V_i}$ for $i \in [3]$. W.l.o.g. we have $n_1 \leq n_2 \leq n_3$. Since $V_1 \dCup V_2 \dCup V_3$ is an equitable partition, it follows that
	\begin{equation*}
		(n_1, n_2, n_3) = \begin{cases}
			(k-1, k, k), & n = 3k-1 \\
			(k,k,k), & n = 3k \\
			(k,k,k+1), & n = 3k+1,
		\end{cases}
	\end{equation*}
	where $k \in \N$ is the unique integer with $\abs{n - 3k}\leq 1$. 
	
	Now, observe that any edge in $H$ covers at least one pair in $\smash{V_1^{(2)} \cup V_2^{(2)} \cup V_3^{(2)}}$ since triples intersecting all the partition classes are not in $E(H)$. In particular, every edge's weight contributes to the (weighted) codegree of at least one such pair, hence
	\begin{alignat*}{3}
		\sum_{e \in E(H)} \varphi(e) &\leq \sum_{p \in V_1^{(2)} \cup V_2^{(2)} \cup V_3^{(2)}} \deg^\varphi(p) &&= \sumOver{i}{1}{3} \binom{n_i}{2} \\
		&= \begin{cases}
			\begin{rcases}
				\frac{(k-1) (k-2)}{2} + 2 \cdot \frac{k (k-1)}{2} , & n = 3k-1 \\[5pt]
				3 \cdot \frac{k(k-1)}{2} , & n = 3k \\[5pt]
				2 \cdot \frac{k(k-1)}{2} + \frac{(k+1)k}{2} , & n = 3k+1
			\end{rcases}
		\end{cases}
		&&= \begin{cases}
			\begin{rcases}
				\frac{(3k-2)(3k-3)}{6}, & n = 3k-1 \\[5pt]
				\frac{3k (3k-3)}{6}, & n = 3k \\[5pt]
				\frac{3k (3k-1)}{6}, & n = 3k+1
			\end{rcases}
		\end{cases} \\
		&< \begin{cases}
			\begin{rcases}
				\frac{(3k-1)(3k-2)}{6}, & n = 3k-1 \\[5pt]
				\frac{3k(3k-1)}{6}, & n = 3k \\[5pt]
				\frac{(3k+1) 3k}{6}, & n = 3k+1
			\end{rcases}
		\end{cases}
		&&= \frac{1}{3} \binom{n}{2}. \text{ \lightning}
	\end{alignat*}
	Therefore, $H$ does not contain a fractional Steiner triple system.
\end{proof}
\begin{corollary}
	$\theta_{\STS}^\ast \geq \theta_{\STS}^f(0) \geq 2/3$.
\end{corollary}
Note that the argument provided here for why $H$ doesn't contain a Steiner triple system relies on a \emph{space barrier}: In some sense, there is not enough room for the edges to accommodate the demand of each pair. This is in contrast to Lee's construction in \cite[Prop.$\ $1.7]{hyunwoo} where it is a \emph{parity argument}. Furthermore, it is important to emphasize that the construction in Proposition \ref{prop:cons2} works for \emph{any} $n \geq 5$. Meanwhile, to construct the hypergraph in \cite[Prop.$\ $1.7]{hyunwoo}, one fundamentally requires $n$ being odd. 

In fact, one can verify by hand that the construction given in \cite[Prop.$\ $1.7]{hyunwoo} contains fractional Steiner triple systems, see Appendix \cite[App. A]{mythesis}, while the construction given in Proposition~\ref{prop:cons2} provably doesn't, which adds to the plausibility of Conjecture \ref{conj:frac_STS}. To solidify this further, it would be productive to find more (presumably) extremal examples for either problem. We note that, if Conjecture \ref{conj:frac_STS} is true, it would be in stark contrast to the \hyperref[conj:nash-williams]{Nash-Williams conjecture} where it is known that the minimum degree threshold for a fractional $K_3$-decomposition is at least $3n/4 + o(n)$, see the concluding remarks in \cite{yuster_fractional_decompositions}. 

\subsection{Higher uniformities}

Another problem would be to consider higher uniformities. Given the previously studied settings, it may be the most natural to consider the minimum codegree threshold for (spanning) Steiner systems with parameters $(r-1, r,n)$\footnote{Meaning that we have $n$ vertices, and every $(r-1)$-set covered by precisely one hyperedge.} in $r$-uniform hypergraphs where $n$ is sufficiently large and satisfies the usual divisibility conditions
\begin{equation*}
	\fa 0\leq i \leq r-1 \colon \left. \binom{r-i}{(r-1)-i} \middle\vert \binom{n-i}{(r-1)-i}\right. .
\end{equation*}
However, with the existence conjecture, whether those necessary divisbility conditions of the parameters is essentially sufficient for the existence of a corresponding design, solved barely a decade ago by \cite{keevash_designs, glock_kühn_lo_osthus}, estimating thresholds in the way we did for Steiner triple systems seems out of reach. 

Still, we wish to provide at least a lower bound which generalizes the construction of Proposition \ref{prop:cons2}, though this time we employ a parity argument. 

\begin{proposition}
	Let $n \geq r \geq 2$. Consider the $r$-uniform hypergraph $H$ with vertex set $V = V_1 \dCup \dots \dCup V_r$, where $\abs{V} = n$, all the parts pairwise differ at most by two in their size and $\abs{V_1}, \dots, \abs{V_{r-1}}$ are odd. Moreover, let the edge set of $H$ be $\smash{E = V^{(r)}} \setminus \set{e \in V^{(r)}\colon \abs{e \cap V_1} = \dots = \abs{e \cap V_r} = 1}$. Then $\delta_{r-1}(H) \geq (r-1)n / r - \mathcal{O}(1)$ and $H$ contains no $(r-1, r, n)$-Steiner system. 
	\label{prop:odd_cons}
\end{proposition}
\begin{proof}
	The first claim $\delta_{r-1}(H) \geq (r-1)n / r - \mO(1)$ follows directly from the fact that all $V_i$'s have roughly equal size. Hence, let us focus on the latter claim: Assume that $S\subseteq H$ is an $(r-1,r, n)$-Steiner system. Consider the sum 
	\begin{equation*}
		M = \sum_{(v_1, \dots, v_{r-1}) \in V_1 \times \dots \times V_{r-1} }\deg_S(v_1, \dots, v_{r-1}) = \abs{V_1} \abs{V_2}  \cdots \abs{V_{r-1}}.
	\end{equation*}
	On the one hand, $M$ must be odd since $V_1, \dots, V_r$ are odd. On the other hand, as we remove the \enquote{partite edges} from $H$, every edge $e \in E(S)$ covering some $(v_1, \dots, v_{r-1}) \in V_1 \times \dots \times V_{r-1}$ must by pigeonhole principle satisfy $\abs{e \cap V_i} = 2$ for a unique $i \in [r-1]$. Let $v_i' \in (e \cap V_i) \setminus \set{v_i}$. It follows that $e$ must contribute to $M$ in precisely two tuples, namely $(v_1, \dots, v_i, \dots ,v_{r-1}),(v_1, \dots, v_i', \dots, v_{r-1})$. Since this is the case for any edge $e$ that contributes to $M$, we must have that $M$ is even, a contradiction. 
\end{proof}

It would also be interesting to study the fractional setting of this problem, i.e. where we are interested in \emph{fractional $(r-1, r, n)$-Steiner systems} $\varphi \colon E(H) \longrightarrow \R_{\geq 0}$ in an $r$-uniform hypergraph $H$ where $\smash{\deg_H^{\varphi}(p)} = 1$ for every $p \in \smash{ V(H)^{(r-1)}}$. Note that the argument for Proposition \ref{prop:cons2} can be used to show that $K_{\floor{n/2} - 1, \ceil{n/2} +1}$ does not contain a fractional perfect matching. Based on this, we conjecture more generally the following:
\begin{conjecture}
	The minimum codegree threshold for a fractional $(r-1, r, n)$-Steiner system is $(r-1) n / r + \Theta(1)$. 
\end{conjecture}

	\section*{Acknowledgement}
The work presented here was developed during the writing of the author's master's thesis, see \cite{mythesis}. The author would like to thank Mathias Schacht for his encouragement to work on this topic, Hyunwoo Lee for his openness to discuss his paper, and Hong Liu for the welcoming stay at the Institute for Basic Science in Daejeon, South Korea.
	\section*{References}
	\printbibliography[heading=none]

@article{hyunwoo,
	author = "Hyunwoo Lee",
	title = "{Towards a high-dimensional Dirac's theorem}",
	journal = "arXiv",
	year = "2023",
	DOI = "10.48550/arXiv.2310.15909"
}

@misc{linial,
	author = "Nati Linial",
	title = "Challenges of high-dimensional combinatorics",
	note = "Lovász's Seventieth Birthday Conference",
	howpublished = {\url{https://www.cs.huji.ac.il/~nati/PAPERS/challenges-hdc.pdf}},
	year = "2018",
	volume = "2"
}

@article{han_person_schacht, 
	title = "On perfect matchings in uniform hypergraphs with large minimum vertex degree",
	author = "Hiep Hàn and Yury Person and Mathias Schacht",
	DOI = "10.1137/080729657",
	year = "2009", 
	journal = "SIAM Journal on Discrete Mathematics", 
	volume = "23", 
	issue = "2",
	pages = "498-504"
}

@article{rödl_rucinski_szemeredi, 
	title = "Perfect matchings in uniform hypergraphs with large minimum degree",
	author = "Vojtěch Rödl and Andrzej Ruciński and Endre Szemerédi",
	DOI = "10.1016/j.ejc.2006.05.008",
	year = "2006", 
	journal = "European Journal of Combinatorics", 
	volume = "27", 
	issue = "8",
	pages = "1333-1349"
}

@article{kühn_osthus_matching, 
	title = "Matchings in hypergraphs of large minimum degree",
	author = "Daniela Kühn and Deryk Osthus",
	DOI = "10.1002/jgt.20139",
	year = "2005", 
	journal = "Journal of Graph Theory", 
	volume = "51", 
	issue = "4",
	pages = "269-280"
}

@article{khan_matching, 
	title = "Perfect matchings in $3$-uniform hypergraphs with large vertex degree",
	author = "Imdadullah Khan",
	DOI = "10.1137/10080796X",
	year = "2011", 
	journal = "SIAM Journal on Discrete Mathematics", 
	volume = "27", 
	issue = "2"
}

@article{treglown_zhao, 
	title = "Exact minimum degree thresholds for perfect matchings in uniform hypergraphs",
	author = "Andrew Treglown and Yi Zhao",
	DOI = "10.1016/j.jcta.2012.04.006",
	year = "2012", 
	journal = "Journal of Combinatorial Theory",
	series = "Series A",
	volume = "119", 
	issue = "7",
	pages = "1500-1522"
}

@article{khan_four_matching, 
	title = "Perfect matchings in $4$-uniform hypergraphs",
	author = "Imdadullah Khan",
	DOI = "10.1016/j.jctb.2015.09.005",
	year = "2016", 
	journal = "Journal of Combinatorial Theory",
	series = "Series B",
	volume = "116", 
	pages = "333-366"
}

@article{rödl_rucinski_szemeredi_co, 
	title = "Perfect matchings in large uniform hypergraphs with large minimum collective degree",
	author = "Vojtěch Rödl and Andrzej Ruciński and Endre Szemerédi",
	DOI = "10.1016/j.jcta.2008.10.002",
	year = "2009", 
	journal = "Journal of Combinatorial Theory", 
	series = "Series A",
	volume = "116", 
	issue = "3",
	pages = "613-636"
}

@article{glock_kühn_lo_osthus, 
	title = "The existence of designs via iterative absorption: Hypergraph $F$-designs for arbitrary $F$",
	author = "Stefan Glock and Daniela Kühn and Allan Lo and Deryk Osthus",
	year = "2023", 
	journal = "Memoirs of the American Mathematical Society",
	volume = "284", 
	DOI = "10.1090/memo/1406",
	pages = "131 pp"
}

@article{keevash_designs,
	title = "The existence of designs",
	author = "Peter Keevash",
	year = "2014", 
	journal = "arXiv", 
	DOI = "10.48550/arXiv.1401.3665"
}

@article{minimalist_designs,
	title = "Minimalist designs",
	author = "Ben Barber and Stefan Glock and Daniela Kühn and Allan Lo and Richard Montgomery and Deryk Osthus", 
	journal = "Random Structures \& Algorithms", 
	year = "2020",
	volume = "57", 
	issue = "1", 
	pages = "47-63",
	DOI = "10.1002/rsa.20915"
}

@article{dirac_original,
	title = "Some theorems on abstract graphs",
	author = "Gabriel Andrew Dirac", 
	journal = "Proceedings of the London Mathematical Society", 
	year = "1952",
	series = "3",
	volume = "2", 
	issue = "1", 
	pages = "69-81",
	DOI = "10.1112/plms/s3-2.1.69"
}

@article{
	delcourt_postle_fractional,
	title = "Progress towards Nash-Williams' conjecture on triangle decompositions",
	author = "Michelle Delcourt and Luke Postle",
	journal = "Journal of Combinatorial Theory",
	series = "Series B",
	year = "2021",
	volume = "146",
	pages = "382-416",
	DOI = "10.1016/j.jctb.2020.09.008"
}

@article{
	dross_fractional, 
	title = "Fractional triangle decompositions in graphs with large minimum degree", 
	author = "François Dross",
	journal = "SIAM Journal on Discrete Mathematics", 
	volume = "30",
	issue = "1", 
	year = "2016",
	pages = "36-42",
	DOI = "10.1137/15M1014310"
}

@article{
	dukes_horsley_fractional,
	title = "On the minimum degree required for a triangle decomposition",
	author = "Peter J. Dukes and Daniel Horsley",
	journal = "SIAM Journal on Discrete Mathematics",
	volume = "34",
	issue = "1", 
	year = "2020",
	DOI = "10.1137/19M1284610"
}

@article{
	barber_fractional_clique,
	title = "Fractional clique decompositions of dense graphs and hypergraphs",
	author = "Ben Barber and Daniela Kühn and Allan Lo and Richard Montgomery and Deryk Osthus",
	journal = "Journal of Combinatorial Theory",
	series = "Series B",
	volume = "127",
	year = "2017",
	DOI = "10.1016/j.jctb.2017.05.005",
	pages = "148-186"
}

@article{nash_williams,
	title = "An unsolved problem concerning decomposition of graphs into
	triangles",
	author = "Crispin St John Alvah Nash-Williams",
	publisher = "North Holland", 
	series = "Combinatorial Theory and its Applications",
	year = "1970",
	issue = "III",
	pages = "1179-1183"
}

@article{
	yuster_fractional_decompositions,
	title = "Asymptotically optimal $K_k$-packings of dense graphs via fractional $K_k$-decompositions",
	author = "Raphael Yuster",
	journal = "Journal of Combinatorial Theory",
	series = "Series B",
	volume = "95",
	issue = "1", 
	year = "2005",
	DOI = "10.1016/j.jctb.2005.02.002",
	pages = "1-11"
}

@Misc{
	lee_personal,
	author       = {Hyunwoo Lee},
	howpublished = {{Personal communication}},
	year         = {2023}
}

@article{
	ravry_zhao,
	author = "Victor Falgas-Ravry and Yi Zhao",
	title = "Codegree thresholds for covering 3-uniform hypergraphs",
	journal = "SIAM Journal on Discrete Mathematics",
	year = "2016", 
	DOI = "10.1137/15M1051452",
	volume = "30",
	issue = "16"
}

@mastersthesis{mythesis,
	title        = {On the existence of Steiner triple systems in 3-uniform hypergraphs},
	author       = {Michael Zheng},
	year         = {2024},
	month        = {June},
	address      = {Hamburg, Germany},
	note         = {Available at \url{https://mx-xz.github.io/assets/pdfs/master_thesis_Michael_Zheng.pdf}},
	school       = {University of Hamburg},
	type         = {Master's thesis}
}
	\appendix 

	\section{Gurobi implementation for (P3)}
	\label{sec:fractional_code}
	Below is an implementation of (P3) as defined above Corollary \ref{cor:8var}. The script was written and run in \textsc{Python} 3.8.10 using the \textsc{Gurobi Optimizer} 10.0.3.
\begin{lstlisting}[language=Python]
# Solving optimization program (P3).

import gurobipy as gp
from gurobipy import GRB

# Create a new model
m = gp.Model("DelcourtPostle")

# Problem is non-convex, so we need to set "NonConvex" accordingly.
m.setParam(GRB.Param.NonConvex, 2)

# (1 - d) is the minimum codegree density. 
# Setting d as follows, the computed objective stays below 1.
# This agrees with our theorem.
d = 0.1421

# Create variables with given lower and upper bounds.
e0 = m.addVar(lb=1.0-d, ub=1.0, name="e_0")
e = m.addVar(lb=1.0-d, ub=1.0, name="e")
f = m.addVar(lb=1.0-d, ub=1.0, name="f")
q0 = m.addVar(name="q_0")
q = m.addVar(name="q")
p = m.addVar(name="p")
r0 = m.addVar(lb=0.5, name="r_0")
r = m.addVar(lb=0, name="r")

# Auxiliary variables representing inverses, 
# Gurobi doesn't directly allow divisions.
e0_inv = m.addVar(name="e_0^-1")
e_inv = m.addVar(name="e^-1")
f_inv = m.addVar(name="f^-1")
q0_inv = m.addVar(name="q_0^-1")
q_inv = m.addVar(name="q^-1")
p_inv = m.addVar(name="p^-1")

# Auxiliary variables for nested terms, 
# Gurobi doesn't directly allow the product of three or more terms.
z0 = m.addVar(name="z_1 * (z_2 + z_3)")
z1 = m.addVar(name="e_0 * r_0")
z2 = m.addVar(name="1 / q_0 * ((1 / e) - (1 / e_0)")
z3 = m.addVar(name="z_4 * z_5")
z4 = m.addVar(name="r / p")
z5 = m.addVar(name="(e^-1 - f^-1) / q + (e^-1 - e_0^-1) / q_0")

# Additional constraints on the common co-neighborhood densities.
m.addConstr(q0 >= e0 + e - 1 - d)
m.addConstr(q0 <= e0)
m.addConstr(q >= e + f - 1 - d)
m.addConstr(q <= e)
m.addConstr(p >= q0 + q - e - d)
m.addConstr(p <= q0)
m.addConstr(r0 <= e0)
m.addConstr(r <= q0)

# Auxiliary constraints for inversion.
m.addConstr(e0_inv * e0 == 1)
m.addConstr(e_inv * e == 1)
m.addConstr(f_inv * f == 1)
m.addConstr(q0_inv * q0 == 1)
m.addConstr(q_inv * q == 1)
m.addConstr(p_inv * p == 1)

# Auxiliary constraints for nested terms inside the objective function.
# z0 will play the role of the objective function.
m.addConstr(z0 == z1 * (z2 + z3))
m.addConstr(z1 == e0 * r0)
m.addConstr(z2 == q0_inv * (e_inv - e0_inv))
m.addConstr(z3 == z4 * z5)
m.addConstr(z4 == r * p_inv)
m.addConstr(z5 == q_inv * (e_inv - f_inv) + q0_inv * (e_inv - e0_inv))

m.setObjective(z0, GRB.MAXIMIZE)

m.optimize()

# Prints out for every variable its value.
for v in m.getVars():
    print(f"{v.VarName} {v.X:g}")

# Prints out the best objective obtained. 
print(f"Obj: {m.ObjVal:g}")
\end{lstlisting}

\end{document}